\documentclass[twoside,11pt]{article}

\usepackage[margin=1in]{geometry}
\usepackage{mathptmx} 
\usepackage{authblk} 
\usepackage{setspace} 
\usepackage{parskip} 
\usepackage{fancyhdr} 
\usepackage{lastpage} 
\usepackage{amsmath,amssymb,amsthm}
\usepackage{bm} 
\usepackage{mathtools} 
\usepackage{physics}

\theoremstyle{plain}
\newtheorem{theorem}{Theorem}[section]

\newtheorem{proposition}[theorem]{Proposition}

\theoremstyle{definition}
\newtheorem{definition}[theorem]{Definition}
\newtheorem{assumption}[theorem]{Assumption}

\newtheorem{remark}[theorem]{Remark}

\usepackage{graphicx}
\usepackage{subcaption} 
\usepackage{booktabs} 
\usepackage{multirow} 
\usepackage{multicol} 
\usepackage{longtable} 
\usepackage{rotating} 
\usepackage{caption} 
\usepackage{float} 
\graphicspath{{figures/}} 
\usepackage[
  backend=biber,       
  sorting=nty,          
  maxnames=3,           
  minnames=1,           
  maxbibnames=5,        
  minbibnames=3,       
  doi=true,             
  isbn=false,           
  url=false,            
  eprint=false          
]{biblatex}
\DeclareFieldFormat{journaltitle}{\textit{#1}}  
\DeclareFieldFormat{booktitle}{\textit{#1}}    
\DeclareFieldFormat{title}{#1.}                

\usepackage[
  colorlinks=true,
  linkcolor=blue,
  citecolor=blue,
  urlcolor=blue,
  pdfauthor={Y. Wang},
  pdftitle={A Convergent Algorithm Based on Deterministic Approximation for a Large Class of Regime-Switching Generalized Stochastic Game-Theoretic Riccati Differential Equations},
  pdfsubject={Iterative Algorithm, Generalized Stochastic Game-Theoretic Riccati Differential Equations},
  pdfkeywords={Iterative Algorithm, Generalized Stochastic Game-Theoretic Riccati Differential Equations}
]{hyperref}
\usepackage{doi} 
\usepackage{url} 

\newcommand{\keywords}[1]{\noindent\textbf{Keywords:} #1}

\title{A Convergent Algorithm Based on Deterministic Approximation for a Large Class of Regime-Switching Generalized Stochastic Game-Theoretic Riccati Differential Equations}
\author{Yiyuan Wang\thanks{Shandong University Zhongtai Securities Institute for Financial Studies, Shandong University, 27 Shanda Nanlu, Jinan, P.R. China, 250100 ({Email: wangyiyuan@mail.sdu.edu.cn}).}}
\date{}

\begin{document}
\maketitle

\begin{abstract}
This paper proposes a novel iterative algorithm to compute the stabilizing solution of regime-switching stochastic game-theoretic Riccati differential equations with periodic coefficients. The method decomposes the original complex stochastic problem into a sequence of deterministic subproblems. By sequentially solving for the minimal solutions of the Riccati differential equations in each subproblem, a sequence of matrix-valued functions is constructed. Leveraging the comparison theorem, the monotonicity, boundedness, and convergence of the iterative sequence are rigorously proven. Numerical experiments verifies algorithm effectiveness and stability. To the best of our knowledge, this is the first general computational approach developed for this class of problems.\\
\keywords{Iterative Algorithm; Generalized Stochastic Game-Theoretic Riccati Differential Equations; Stabilizing Solutions; Deterministic Approximation} \\
\end{abstract}

\section{Introduction} 
\label{sec:introduction}

Infinite-horizon zero-sum linear-quadratic stochastic differential games (ZSLQSDG) with Markovian regime-switching model adversarial decision-making under random environmental shifts—such as in financial markets or cyber-physical systems. The state dynamics follow an It\^o differential equation driven by multiplicative noise, with coefficients that switch according to a finite-state continuous-time Markov chain. This setup leads to a coupled system of generalized stochastic game-theoretic Riccati differential equations (GTRDEs)—one for each regime—interconnected through the Markov chain's transition rates. (See \cite*{Basar1999, Mou2006, Michael2006, Ungureanu2012,Dragan2020,Sun2020_book,Lv2020,Li2021,Liu2024,Li2024,Lv2025} and related references for in-depth discussions on relevant issues.)

In the infinite-horizon setting, the central object of interest is the stabilizing solution—a set of constant, positive semi-definite matrices that yield equilibrium feedback strategies ensuring both mean-square exponential stability and equilibrium performance. While existence and uniqueness are well established under standard assumptions (e.g., \cite{Dragan2020}), computing this stabilizing solution remains challenging. The system is highly nonlinear, globally coupled, and lacks natural boundary conditions, making classical initial- or terminal-value solvers unsuitable. Moreover, the two-layer iterative algorithms developed for specific configurations in works such as \cite*{Lanzon2008, Dragan2008, Feng2010, Dragan2011, Dragan201501, Dragan201502, Ivanov2015, Ivanov2018} are no longer applicable in this general case. 

\Cite{Dragan2017} proposed an iterative method based on a deterministic approximation to compute the stabilizing solution of a broad class of stochastic GTRDEs arising in $H_\infty$ control problems. \Cite{Aberkane2023} studied numerical algorithms for generalized Riccati difference equations that arise in linear-quadratic stochastic difference games. To the best of our knowledge, however, no existing algorithm is capable of fully computing the stabilizing solution of the coupled generalized time-varying Riccati equations associated with infinite-horizon zero-sum linear-quadratic differential games featuring regime-switching dynamics and multiplicative noise.

This paper fills this gap by proposing a novel numerical scheme that decomposes the stochastic problem into a sequence of deterministic subproblems. Under natural assumptions and for periodic coefficients, the algorithm proceeds by constructing a sequence of matrix-valued functions through the successive solution of minimal solutions to the associated deterministic GTRDEs at each step. Using the comparison theorem, we establish that the resulting iterative sequence is monotonic and bounded, and converging to the stabilizing solution of the original problem. Numerical experiments confirm the efficacy of the proposed approach.
Importantly, our method unifies and extends several classical settings: classical ZSLQSDG-type GTRDEs, $H_\infty$-type stochastic GTRDEs with regime-switching, and the more general GTRDEs corresponding arise in ZSLQSDGs with regime-switching. This offers a unified computational framework for a broad family of adversarial control problems under uncertainty.

Remaining content outline: Section \ref{sec:preliminary} covers infinite-horizon ZSLQSDG with regime-switching mathematical framework, parameter-decoupled auxiliary problems and some useful preliminary results. Section \ref{sec:main_results} presents main results: algorithm iteration process, associated deterministic problem formulation, proofs of algorithm convergence. Section \ref{sec:numerical_experiments} details numerical experiments.

\section{Preliminary} 
\label{sec:preliminary}

\subsection{Notation}
To facilitate reading, we summarize the key mathematical symbols used in this work.
\begin{itemize}
\item $\mathbb{R}_+$: Set of non-negative real numbers; $\mathbb{R}^{p \times q}$: Set of all $p \times q$ real matrices; $\mathbb{S}_p$: Set of all $p \times p$ real symmetric matrices; $\mathbb{S}^N_p$: Set of $N$-tuples of $p \times p$ real symmetric matrices, i.e., $\mathbb{S}^N_p = (\mathbb{S}_p, \dots, \mathbb{S}_p)$ (with $N$ components); $\mathbb{I}_p$: $p \times p$ identity matrix.
\item $\top$: Matrix transpose; $(\cdot)^{-1}$: Matrix inverse (for invertible matrices). $\det(\cdot)$: Matrix determinant; $\mathrm{sgn}(\cdot)$: Matrix sign (characterized by its inertia, i.e., the number of positive/negative eigenvalues).
\item If $A\in\mathbb{S}_{p+}$ (resp., $A\in\overline{\mathbb{S}}_{p+}$) is a symmetric positive definite (resp., symmetric positive semi-definite) matrix, we write $A\succ0$ (resp., $A\succeq0$). For any $A,B\in\mathbb{S}_p$, we use the notation $A\succ B$ (resp., $A\succeq B$) to indicate that $A - B\succ0$ (resp., $A - B\succeq0$).
\item $\mathbb{E}[\cdot]$: Expectation with respect to the probability measure $\mathbb{P}$; $\langle \cdot, \cdot \rangle$: The inner product on a Hilbert space.
\end{itemize}

\subsection{A Class of Regime-Switching Stochastic Game-Theoretic Riccati Differential Equations}
\label{subsec:gtrdes}

Consider the following system of nonlinear differential equations:
\begin{equation}
\label{pf:gtrde}
\begin{aligned}
&\frac{d}{dt}{X}(t,i) +A_0^{\top}(t,i)X(t,i) + X(t,i)A_0(t,i) + \sum_{k=1}^r A_k^{\top}(t,i)X(t,i)A_k(t,i) +\sum_{j=1}^N q_{ij}X(t,j) + M(t,i) \\
&- \left( X(t,i)B_0(t,i) + \sum_{k=1}^r A_k^{\top}(t,i)X(t,i)B_k(t,i) + L(t,i) \right) \left( R(t,i) + \sum_{k=1}^r B_k^{\top}(t,i)X(t,i)B_k(t,i) \right)^{-1} \\
&\quad \left( B_0^{\top}(t,i)X(t,i) + \sum_{k=1}^r B_k^{\top}(t,i)X(t,i)A_k(t,i) + L^{\top}(t,i) \right), t \geq 0 \, ; 1 \leq i \leq N,
\end{aligned}    
\end{equation}
with the unknown function $t \in \mathbb{R}_+ \mapsto \mathbb{X}(t) = (X(t,1), \ldots, X(t,N)) \in \mathbb{S}^N_n$.

We are interested in a class of Stochastic GTRDEs arising in ZSLQSDG with Regime-Switching. The admissible solution set for the Stochastic GTRDEs \eqref{pf:gtrde} consists of all mappings $\mathbb{X}(\cdot): \mathbb{R}_+ \to \mathbb{S}^n$ satisfying the following sign condition:
\begin{equation}
\label{pf:sign_conditions_1}
\mathrm{sgn}\left( R(t,i) + \sum_{k=1}^r B_k^{\top}(t,i)X(t,i)B_k(t,i) \right) = \mathrm{sgn} \, \mathrm{diag}\left(-I_{m_1}, I_{m_2}\right) , \quad \forall \, (t, i) \in  \mathbb{R}_{+} \times \mathfrak{N}, 
\end{equation}
where $\mathfrak{N} = \{1, 2, \ldots, N\}$ and $\mathrm{diag}\left(-I_{m_1}, I_{m_2}\right)$ is block diagonal matrix with $-I_{m_1}$ and $I_{m_2}$ as diagonal blocks. To ensure the problem is well-defined, we impose the following parameter constraints on the Stochastic GTRDEs \eqref{pf:gtrde}.

\begin{assumption}
\label{pf:ass_1}
\begin{enumerate}
    \,
    \item[(1)] In Stochastic GTRDEs \eqref{pf:gtrde}, $t \mapsto A_k(t,i) : \mathbb{R}_+ \to \mathbb{R}^{n \times n}; t \mapsto B_k(t,i) : \mathbb{R}_+ \to \mathbb{R}^{n \times m}(0 \leq k \leq r) \, , t \mapsto L(t,i) : \mathbb{R}_+ \to \mathbb{R}^{n \times m},\, t \mapsto M(t,i) : \mathbb{R}_+ \to \mathbb{S}_n,\, t \mapsto R(t,i) : \mathbb{R}_+ \to \mathbb{S}_m$ are continuous matrix-valued functions which are periodic with period $\theta > 0$.
    \item[(2)] The scalars $q_{ij}$ involved in Stochastic GTRDEs \eqref{pf:gtrde} have the properties
    \begin{equation}
    \label{pf:ass_12}
    q_{ij} \geq 0, \, i \neq j; \quad \sum_{j=1}^N q_{ij} = 0, \, 1 \leq i \leq N.
    \end{equation}
    \item[(3)]For all $(t,i) \in \mathbb{R}_+ \times \mathfrak{N}$, the following matrix inequalities hold:
    \begin{equation}
    \label{pf:ass_13}
    \begin{aligned}
    &R_{22}(t,i) \succ 0, M(t,i) - L_2(t,i) R_{22}(t,i)^{-1} L_2^{\top}(t,i) \succeq 0
    \end{aligned}
    \end{equation}
    where the block matrices are defined as:
    \[
    R(t,i) = \begin{pmatrix} R_{11}(t,i) & R_{12}(t,i) \\ R_{12}^{\top}(t,i) & R_{22}(t,i) \end{pmatrix}, \quad 
    L(t,i) = \begin{pmatrix} L_1(t,i) & L_2(t,i) \end{pmatrix},
    \]
    with $R_{l_1l_2}(t,i) \in \mathbb{R}^{m_{l_1} \times m_{l_2}} (l_1, l_2 = 1, 2)$ and $L_l(t,i) \in \mathbb{R}^{n \times m_l}(l = 1, 2)$. Here, $m_1, m_2 \geq 1$ and $m_1 + m_2 = m$ are given constants.
\end{enumerate}
\end{assumption}

\begin{remark}
\Cite{Dragan2020} established sufficient conditions guaranteeing the existence of a bounded, stabilizing solution to this class of Riccati differential equations, while further deriving its correspondence with ZSLQZSDG (see Section 3 of their work for details). The parameter configurations employed herein align with the assumptions put forth in \cite{Dragan2020}.
By leveraging Lemma 2 from Chapter 4 of \cite{Halanay1994}, the following statements are equivalent:
\begin{enumerate}
    \item[i.] The solution $\mathbb{X}(\cdot)$ satisfies the sign condition \ref{pf:sign_conditions_1};  
    \item[ii.] The solution $\mathbb{X}(\cdot)$ satisfies the sign conditions:  
    \begin{equation}
    \label{pf:sign_conditions_21}   
    \begin{aligned}
    &\mathbb{R}_{22}(t)[\mathbb{X}](i) = R_{22}(t,i) + \sum_{k=1}^{r} B_{k2}^{\top}(t,i) X(t,i) B_{k2}(t,i) \succ 0, \quad \forall \, (t,i) \in \mathbb{R}_{+} \times \mathfrak{N};
    \end{aligned}
    \end{equation}
    \begin{equation}
    \label{pf:sign_conditions_22}   
    \begin{aligned}
    &\mathbb{R}_{22}^{\sharp}(t)[\mathbb{X}](i) = R_{11}(t,i) + \sum_{k=1}^{r} B_{k1}^{\top}(t,i) X(t,i) B_{k1}(t,i) - \left( R_{12}(t,i) + \sum_{k=1}^{r} B_{k1}^{\top}(t,i) X(t,i) B_{k2}(t,i) \right) \\
    &\quad \times \mathbb{R}_{22}(t)[\mathbb{X}](i)^{-1} \left( R_{12}(t,i) + \sum_{k=1}^{r} B_{k1}^{\top}(t,i) X(t,i) B_{k2}(t,i) \right)^{\top} \prec 0, \quad \forall \, (t,i) \in \mathbb{R}_{+} \times \mathfrak{N},
    \end{aligned}
    \end{equation}
    where $\mathbb{X}(\cdot)$ denotes a bounded, $\theta$-periodic solution to the Stochastic GTRDEs \eqref{pf:gtrde} with $X(t,i)\in\overline{\mathbb{S}}_{n+}$ for all $(t,i) \in \mathbb{R}_{+} \times \mathfrak{N}$.
\end{enumerate}
\end{remark}

This type of Stochastic GTRDEs is derived from the following problem. Let $(\Omega, \mathcal{F}, \{\mathcal{F}_t\}_{t \geq 0}, \mathbb{P})$ denote a filtered probability space satisfying the usual conditions, which supports an $r$-dimensional standard Brownian motion $W(t) = (w_1(t), \ldots, w_r(t))^\top$ and a standard right-continuous Markov chain $\eta_t$ with finite state space $\mathfrak{N}$. The Markov chain has a transition semigroup $P(t) = e^{Qt}$ for $t \geq 0$, where $Q = (q_{ij})_{N \times N}$ is the generator matrix satisfying conditions \eqref{pf:ass_12}. Assume that $W(t)$ and $\eta_t$ are independent, and the initial distribution of $\eta_t$ satisfies $\mathbb{P}(\eta_0 = i) = \pi_0(i) > 0$ for every $i \in \mathfrak{N}$. The state of the system evolves according to the controlled linear stochastic differential equation on $[0, \infty)$:
\begin{equation}
\label{pf:lqzsdg_sde}
\begin{cases}
dx(t) = \left( A_0(t,\eta_t)x(t) + B_{01}(t,\eta_t)u_1(t) + B_{02}(t,\eta_t)u_2(t) \right) dt \\
\quad + \sum_{k=1}^{r} \left( A_k(t,\eta_t)x(t) + B_{k1}(t,\eta_t)u_1(t) + B_{k2}(t,\eta_t)u_2(t) \right) dw_k(t) \\
x(0) = x_0
\end{cases},
\end{equation}
where $x_0 \in \mathbb{R}^n$, and the control inputs $u_1(t) \in \mathbb{R}^{m_1 \times n}$ and $u_2(t) \in \mathbb{R}^{m_2 \times n}$ are defined as $\begin{pmatrix} u_1(\cdot) \\ u_2(\cdot) \end{pmatrix} = u(\cdot)$. For each $i \in \mathfrak{N}$ and $k \in \{0, 1, \dots, r\}$, the matrix $B_{kl}(\cdot, i) \in \mathbb{R}^{n \times m_l}(l = 1, 2)$ is given by
$B_k(\cdot, i) = \begin{pmatrix} B_{k1}(\cdot, i) & B_{k2}(\cdot, i) \end{pmatrix}.$

Player 1 and Player 2 share the same performance functional:
\begin{equation}
\label{pf:lqzsdg_pf}
    J(x_0;u_1,u_2)\triangleq\mathbb{E}\int_{0}^{\infty}\Bigg[\Bigg\langle\begin{pmatrix}
    M(t,\eta_t) & L_1(t,\eta_t)  & L_2(t,\eta_t) \\
    L_1^{\top}(t,\eta_t) & R_{11}(t,\eta_t) & R_{12}(t,\eta_t)\\
    L_2^{\top}(t,\eta_t) & R_{12}^{\top}(t,\eta_t) & R_{22}(t,\eta_t)
    \end{pmatrix}\begin{pmatrix}
    x_u(t)\\
    u_1(t)\\
    u_2(t)
    \end{pmatrix},\begin{pmatrix}
    x_u(t)\\
    u_1(t)\\
    u_2(t)
    \end{pmatrix}\Bigg\rangle
    \Bigg]dt,
\end{equation}
where $x_u(\cdot)$ is the solution of the initial-value problem \eqref{pf:lqzsdg_sde} determined by the input $u(\cdot)$. The class of $\mathcal{U}_{\text{adm}}(x_0)$ of the admissible controls consists of all stochastic processes, $u = \{u(t)\}_{t \geq 0} \in L_{\mathcal{\eta},w}^2(\mathbb{R}_+; \mathbb{R}^m)$ with the property that $x_u(\cdot) \in L_{\mathcal{\eta},w}^2(\mathbb{R}_+; \mathbb{R}^n)$ and $\lim_{t \to \infty} \mathbb{E}[|x_u(t)|^2 \big| \eta_0 = i] = 0$ for every $i \in \mathfrak{N}$. \footnote{For precise definition of the spaces $L_{\mathcal{\eta},w}^2(\mathbb{R}_+; \mathbb{R}^d)$ of stochastic processes adapted to the filtration generated by $\{(\eta_t, w(t))\}_{t \geq 0}$, refer to Chapter 1 of \cite{Dragan2013book}.}

In this zero-sum game, Player 1 (\textit{the maximizer}) selects control $u_1$ to maximize \eqref{pf:lqzsdg_pf}, while Player 2 (\textit{the minimizer}) chooses $u_2$ to minimize the same function. The problem is to find an admissible control pair $(\tilde{u}_1,\tilde{u}_2)$ such that
\[
J(x_0; u_1, \tilde{u}_2) \leq J(x_0; \tilde{u}_1, \tilde{u}_2) \leq J(x_0; \tilde{u}_1, u_2)
\]
holds for all admissible $(u_1, u_2)$.
To facilitate subsequent discussions, we define the core notation set related to the Stochastic GTRDEs \eqref{pf:gtrde} as $\boldsymbol{\Sigma}= (\hat{\mathbb{A}}(\cdot),\mathbb{B}_0(\cdot),\Pi_1(\cdot),\Pi_2(\cdot),\Pi_3(\cdot),\mathbb{Q}(\cdot))$ and its components as follows:  
\begin{itemize}
    \item For all $t \in \mathbb{R}_+ : \hat{\mathbb{A}}(t) =(\hat{A}(t,1), \ldots, \hat{A}(t,1))= (A_0(t,1)+\frac{1}{2}q_{11}\mathbb{I}_{n}, \ldots, A_0(t,N)+\frac{1}{2}q_{NN}\mathbb{I}_{n}) ,\, \mathbb{B}_0(t) = (B_0(t,1), \ldots, B_0(t,N)) $;
    \item For all $t \in \mathbb{R}_+$, the operator $\Pi_j(t)(j = 1, 2, 3)$ is defined on $S_n^N$ by 
    \[
    \Pi_j(t)[\mathbb{X}] = (\Pi_j(t)[\mathbb{X}](1), \ldots, \Pi_j(t)[\mathbb{X}](N))
    \]
    where for each $i \in \mathfrak{N}$:
    \[
    \begin{cases}
    \Pi_1(t)[\mathbb{X}](i) = \sum_{l=1, l \neq i}^N q_{il} X(l) + \sum_{k=1}^r A_k^{\top}(t, i) X(i) A_k(t, i) ; \\
    \Pi_2(t)[\mathbb{X}](i) = \sum_{k=1}^r A_k^{\top}(t, i) X(i) B_k(t, i) ; \\
    \Pi_3(t)[\mathbb{X}](i) = \sum_{k=1}^r B_k^{\top}(t, i) X(i) B_k(t, i) .
    \end{cases}
    \]
    \item For all $t \in \mathbb{R}_+$, let $\mathbb{Q}(t) = \left( Q(t,1),\ \ldots,\ Q(t,N) \right) \in S_{n+m}^N$, where each symmetric block is given by
    \[
    Q(t,i) = 
    \begin{pmatrix}
    M(t,i) & L(t,i) \\
    L^\top(t,i) & R(t,i)
    \end{pmatrix} \, \text{for each} \,\, i \in \mathfrak{N}.
    \]
\end{itemize}

\subsection{Some Useful Preliminary Results}

\begin{definition}
A global solution $\tilde{\mathbb{X}}(\cdot): \mathbb{R}_+ \to \mathbb{S}_n^N$ of the Stochastic GTRDEs \eqref{pf:gtrde} is called a \emph{stabilizing solution} if it satisfies the following two conditions:
\begin{enumerate}
    \item[(1)] For each $i \in \mathfrak{N}$,
    \[
    \inf_{t \in \mathbb{R}_+} | \det (R(t,i) + \sum_{k=1}^r B_k^{\top}(t,i)\tilde{X}(t,i)B_k(t,i)) | > 0;
    \]
    \item[(2)] The closed-loop system
    \[
    dx(t) = \left( A_0(t,\eta_t) + B_0(t,\eta_t) F_{\tilde{X}}(t,\eta_t) \right)x(t)\,dt + \sum_{k=1}^r \left( A_k(t,\eta_t) + B_k(t,\eta_t)F_{\tilde{X}}(t,\eta_t) \right)x(t)\,dw_k(t)
    \]
    is such that its zero solution is exponentially stable in mean square (ESMS). For conciseness, we say the system
    \[
    \left( \mathbb{A}_0(\cdot) + \mathbb{B}_0(\cdot)\mathbb{F_{\tilde{X}}}(\cdot), \dots, \mathbb{A}_r(\cdot) + \mathbb{B}_r(\cdot)\mathbb{F_{\tilde{X}}}(\cdot) \right)
    \]
    is stable. Here, the feedback operator $\mathbb{F_X}(\cdot) = (F_X(\cdot,1), \dots, F_X(\cdot,N))$ with $F_X(t,i)\in R^{m \times n} =$
    \[
    -( R(t,i) + \sum_{k=1}^r B_k^{\top}(t,i)X(t,i)B_k(t,i) )^{-1} ( B_0^{\top}(t,i)X(t,i) + \sum_{k=1}^r B_k^{\top}(t,i)X(t,i)A_k(t,i) + L^{\top}(t,i) ).
    \]
\end{enumerate}
\end{definition}

\begin{theorem}[\cite*{Dragan2020}]
\label{comparison_theorem}
Let $\mathbb{X}^2(\cdot) : I_2 \subseteq \mathbb{R}_+ \to S^N_n$ be a solution of the Stochastic GTRDEs \eqref{pf:gtrde} associated to the quintuple $ \Sigma^2 = (\mathbb{\hat{A}}_0(\cdot), \mathbb{B}_0(\cdot), \Pi_1(\cdot), \Pi_2(\cdot), \Pi_3(\cdot), \mathbb{Q}^2(\cdot))$ that verifies the sign condition \ref{pf:sign_conditions_21},\ref{pf:sign_conditions_22} and let $\mathbb{X}^1(\cdot) : I_1 \subseteq \mathbb{R}_+ \to S^N_n$ be a solution of the Stochastic GTRDEs \eqref{pf:gtrde} associated to the quintuple $ \Sigma^1 = (\mathbb{\hat{A}}_0(\cdot), \mathbb{B}_0(\cdot), \\\Pi_1(\cdot), \Pi_2(\cdot), \Pi_3(\cdot), \mathbb{Q}^1(\cdot))$ that verifies the sign condition \ref{pf:sign_conditions_21}.
Assume that $\mathbb{Q}^2(t,i) \succeq \mathbb{Q}^1(t,i)$ for all $t \in I_1 \cap I_2$ and $i \in \mathfrak{N}$, where 
\[
\mathbb{Q}^{u}(t,i) = \begin{pmatrix} M^{u}(t,i) & L^{u}(t,i) \\ L^{u\top}(t,i) & R^{u}(t,i) \end{pmatrix}, u = 1,2. 
\] 
If there exists $\tau \in I_1 \cap I_2$ such that $X^2(\tau,i) \succeq X^1(\tau,i)$ for each $i \in \mathfrak{N}$ , then $X^2(t,i) \succeq X^1(t,i)$ for all $t \in [0, \tau] \cap (I_1 \cap I_2)$ and $i \in \mathfrak{N}$.
\end{theorem}

A useful reformulation is obtained by parameterizing the control of Player 2 as a linear feedback of the state and the control of Player 1.
Setting $u_2(t) = K(t,\eta_t)x(t) + W(t,\eta_t)u_1(t)$ in \eqref{pf:lqzsdg_sde} and \eqref{pf:lqzsdg_pf}, we obtain
\begin{equation}
\label{pr:lqzsdg_kw_sde}
\begin{cases}
dx(t) = \left( A_{0K}(t,\eta_t)x(t) + B_{0W}(t,\eta_t)u_1(t) \right) dt + \sum_{k=1}^r \left( A_{kK}(t,\eta_t)x(t) + B_{kW}(t,\eta_t)u_1(t) \right) dw_k(t)\\
x(0) = x_0 \in \mathbb{R}^n    
\end{cases},   
\end{equation}
in which $x_0 \in \mathbb{R}^n$, and 
\begin{equation}
\label{pr:lqzsdg_kw_pf}
    J_{\text{KW}}(x_0, u_1)
    \triangleq\mathbb{E}\int_{0}^{\infty}\Bigg[\Bigg\langle
    \begin{pmatrix} M_K(t,\eta_t) & L_{\text{KW}}(t,\eta_t) \\ L_{\text{KW}}^{\top}(t,\eta_t) & R_W(t,\eta_t) \end{pmatrix}
    \begin{pmatrix}
    x_{u_1}(t)\\
    u_1(t)
    \end{pmatrix},
    \begin{pmatrix}
    x_{u_1}(t)\\
    u_1(t)
    \end{pmatrix}\Bigg\rangle
    \Bigg]dt,
\end{equation}
where $x_{u_1}(t)$ is the solution of \eqref{pr:lqzsdg_kw_sde} corresponding to input $u_1(t)$ and
\[
\begin{cases}
A_{kK}(t,i) &= A_k(t,i) + B_{k2}(t,i)K(t,i),\quad 0 \leq k \leq r , 1 \leq i \leq N;\\
B_{kW}(t,i) &= B_{k1}(t,i) + B_{k2}(t,i)W(t,i), \quad 0 \leq k \leq r, 1 \leq i \leq N ;\\
M_K(t,i) &= M(t,i) + L_{2}(t,i)K(t,i) + K^{\top}(t,i)L_{2}^{\top}(t,i) + K^{\top}(t,i)R_{22}(t,i)K(t,i), 1 \leq i \leq N ;\\
L_{\text{KW}}(t,i) &= L_{1}(t,i) + K^{\top}(t,i)R_{12}^{\top}(t,i) + (L_{2}(t,i) + K^{\top}(t,i)R_{22}(t,i))W(t,i), 1 \leq i \leq N ;\\
R_W(t,i) &= \begin{pmatrix} I_{m_1} \\ W(t,i) \end{pmatrix}^{\top} \begin{pmatrix} R_{11}(t,i) & R_{12}(t,i) \\ R_{12}^{\top}(t,i) & R_{22}(t,i) \end{pmatrix} \begin{pmatrix} I_{m_1} \\ W(t,i) \end{pmatrix}, 1 \leq i \leq N .
\end{cases} 
\]
The corresponding Riccati differential equations is
\begin{equation}
\label{pr:gtrde_kw}
\begin{split}
& \frac{d}{dt}{Y}(t,i) +A_{0K}^{\top}(t,i)Y(t,i) + Y(t,i)A_{0K}(t,i) + \sum_{k=1}^r A_{kK}^{\top}(t,i)Y(t,i)A_{kK}(t,i) + \sum_{j=1}^N q_{ij}Y(t,j)+ M_K(t,i) \\
&- \left( Y(t,i)B_{0W}(t,i) + \sum_{k=1}^r A_{kK}^{\top}(t,i)Y(t,i)B_{kW}(t,i) + L_{\text{KW}}(t,i) \right) \left( R_W(t,i) + \sum_{k=1}^r B_{kW}^{\top}(t,i)Y(t,i)B_{kW}(t,i) \right)^{-1} \\
& \quad \left( B_{0W}^{\top}(t,i)Y(t,i) + \sum_{k=1}^r B_{kW}^{\top}(t,i)Y(t,i)A_{kK}(t,i) + L_{\text{KW}}^{\top}(t,i) \right) =0, t \ge 0; 1 \leq i \leq N.
\end{split}   
\end{equation}

Similarly, we define the core notation set associated with the Riccati differential equations \eqref{pr:gtrde_kw} as $\boldsymbol{\Sigma}_{\mathbb{KW}} = \left( \mathbb{\hat{A}}_{0\mathbb{K}}(\cdot),\ \mathbb{B}_{0\mathbb{W}}(\cdot),\ \Pi_{1\mathbb{K}}(\cdot),\ \Pi_{2\mathbb{KW}}(\cdot),\ \Pi_{3\mathbb{W}}(\cdot),\ \mathbb{Q}_{\mathbb{KW}}(\cdot) \right)$. For all $t \in \mathbb{R}_+$, the each component specified as follows:
\begin{itemize}
    \item $\mathbb{\hat{A}}_{0\mathbb{K}}(t) =\left( \hat{A}_{0K}(t,1),\ldots,\hat{A}_{rK}(t,1)  \right)= \left( A_{0K}(t,1) + \frac{1}{2}q_{11} \mathbb{I}_n,\ \ldots,\ A_{0K}(t,N) + \frac{1}{2}q_{NN} \mathbb{I}_n \right),\, \mathbb{B}_{0\mathbb{W}}(t) =\\ \left( B_{0W}(t,1),\ \ldots,\ B_{0W}(t,N) \right)$.
    \item The operators $\Pi_{1\mathbb{K}}(t), \Pi_{2\mathbb{KW}}(t)$ and $\Pi_{3\mathbb{W}}(t)$ are defined on $S_n^N$  such that for any  $\mathbb{X} \in S_n^N$:
    \[
    \Pi_{j}(t)[\mathbb{X}] = (\Pi_{j}(t)[\mathbb{X}](1), \ldots, \Pi_{j}(t)[\mathbb{X}](N)), \quad j \in \{1\mathbb{K},\ 2\mathbb{KW},\ 3\mathbb{W}\},
    \]
    where for each $i \in \mathfrak{N}$:
    \[
    \begin{cases}
    \Pi_{1\mathbb{K}}(t)[\mathbb{X}](i) = \sum_{l=1, l \neq i}^N q_{il} X(l) + \sum_{k=1}^r A_{kK}^{\top}(t, i) X(i) A_{kK}(t, i), \\
    \Pi_{2\mathbb{KW}}(t)[\mathbb{X}](i) = \sum_{k=1}^r A_{kK}^{\top}(t, i) X(i) B_{kW}(t, i), \\
    \Pi_{3\mathbb{W}}(t)[\mathbb{X}](i) = \sum_{k=1}^r B_{kW}^{\top}(t, i) X(i) B_{kW}(t, i).
    \end{cases}
    \]
    \item $\mathbb{Q}_{\mathbb{KW}}(t) = \left( Q_{KW}(t,1),\ \ldots,\ Q_{KW}(t,N) \right) \in S_{n+m}^N$, where each symmetric block takes the form:
     \[
    Q_{KW}(t,i) = 
    \begin{pmatrix}
    M_K(t,i) & L_{KW}(t,i) \\
    L_{KW}^\top(t,i) & R_W(t,i)
    \end{pmatrix}, \, \text{for each}\,\, i \in \mathfrak{N}.
    \]    
\end{itemize}

\begin{remark}
\label{pr:pi_kw}
From the following decomposition, it is not difficult to obtain that: for any $(t, i) \in \mathbb{R}_+ \times \mathfrak{N}, \, X(t,i) \succeq 0$, we have 
\begin{align*}
    &\begin{pmatrix} \Pi_{1}(t)[\mathbb{X}](i) & \Pi_{2}(t)[\mathbb{X}](i) \\ \Pi_{2}(t)[\mathbb{X}](i)^{\top} & \Pi_{3}(t)[\mathbb{X}](i) \end{pmatrix}=\begin{pmatrix} \sum_{l=1, l \neq i}^N q_{il} X(l) & 0 \\ 0 & 0 \end{pmatrix} + \sum_{k=1}^r \begin{pmatrix} A^{\top}_{k}(t, i) \\ B^{\top}_{k}(t, i) \end{pmatrix} X(t,i) \begin{pmatrix} A_{k}(t, i) & B_{k}(t, i) \end{pmatrix} \succeq 0, 
\end{align*}
and
\begin{align*}
    &\begin{pmatrix} \Pi_{1\mathbb{K}}(t)[\mathbb{X}](i) & \Pi_{2\mathbb{KW}}(t)[\mathbb{X}](i) \\ \Pi_{2\mathbb{KW}}(t)[\mathbb{X}](i)^{\top} & \Pi_{3\mathbb{W}}(t)[\mathbb{X}](i) \end{pmatrix}=\begin{pmatrix} \sum_{l=1, l \neq i}^N q_{il} X(l) & 0 \\ 0 & 0 \end{pmatrix}+ \sum_{k=1}^r J_{k\mathbb{KW}}(t,i)^{\top} X(t,i) J_{k\mathbb{KW}}(t,i) \succeq 0,\\ 
    &\text{where} \,\, J_{k\mathbb{KW}}(t,i)=\begin{pmatrix} A_{k}(t, i)& B_{k1}(t, i)&B_{k2}(t, i)\end{pmatrix} \begin{pmatrix} \mathbb{I}_{n} & 0 \\ 0 & \mathbb{I}_{m_1} \\  \mathbb{K}(t,i) &\mathbb{W}(t,i)\end{pmatrix},\, 0 \leq k \leq r.
\end{align*}
\end{remark}

The concept of the stabilizing solution for Riccati differential equations \eqref{pr:gtrde_kw} is defined analogously to that for the Stochastic GTRDEs \eqref{pf:gtrde}. In this case, the system
$\left( \mathbb{A}_0(\cdot) + \mathbb{B}_0(\cdot)\mathbb{L_{\tilde{Y}}}(\cdot),  \dots, \mathbb{A}_r(\cdot) + \mathbb{B}_r(\cdot)\mathbb{L_{\tilde{Y}}}(\cdot) \right)$ is stable, where $\mathbb{\tilde{Y}}(\cdot)$ is the stabilizing solution of Riccati differential equations \eqref{pr:gtrde_kw} and the feedback gain $\mathbb{L_{\tilde{Y}}}(t) = (L_{\tilde{Y}}(t,1), \dots, L_{\tilde{Y}}(t,N))$ with $L_{\tilde{Y}}(t,i) \in R^{m \times n}=$
\begin{equation}
\label{pr:feedback_gain_l}
-( R_W(t,i) + \sum_{k=1}^r B_{kW}^{\top}(t,i)\tilde{Y}(t,i)B_{kW}(t,i) )^{-1} ( B_{0W}^{\top}(t,i)\tilde{Y}(t,i) + \sum_{k=1}^r B_{kW}^{\top}(t,i)\tilde{Y}(t,i)A_{kK}(t,i) + L_{\text{KW}}^{\top}(t,i) ),
\end{equation}
for all $(t,i) \in \mathbb{R}_+ \times \mathfrak{N}$.

Throughout this work, $\mathcal{A}^{\Sigma}$ denotes the set of all pairs $(\mathbb{K}(\cdot), \mathbb{W}(\cdot))$ associated with system $\Sigma$, where $\mathbb{K}(\cdot) = (K(\cdot,1),\dots,K(\cdot,N))$, $\mathbb{W}(\cdot) = (W(\cdot,1),\dots,W(\cdot,N))$ with $K(\cdot,i) : \mathbb{R}_+ \to \mathbb{R}^{m_2 \times n}$, $W(\cdot,i) : \mathbb{R}_+ \to \mathbb{R}^{m_2 \times m_1}$ ($1 \leq i \leq N$) being $\theta$-periodic continuous functions, satisfying:
\begin{enumerate}
    \item[(1)] The system $(\mathbb{A}_{0\mathbb{K}},\dotsb,\mathbb{A}_{r\mathbb{K}})$ is stable;
    \item[(2)] The corresponding Riccati differential equations \eqref{pr:gtrde_kw} has a bounded stabilizing solution $\tilde{Y}_{\mathbb{KW}}(\cdot)$, satisfying the sign conditions
    \begin{equation}
    \label{sign_condition_31}
    R_{\mathbb{W}}(t, i) + \sum_{k=1}^r B_{kW}^{\top}(t, i)\tilde{Y}_{\mathbb{K}\mathbb{W}}(t, i)B_{kW}(t, i) \preceq 0, \forall (t, i) \in \mathbb{R}_+ \times \mathfrak{N}.
    \end{equation}
\end{enumerate}

\begin{remark}
If the set $\mathcal{A}^\Sigma$ is non-empty and $R_{22}(t,i) \succ 0; \forall (t,i) \in \mathbb{R}_+ \times \mathfrak{N} $, we have 
\begin{equation*}
R_{11}(t,i) - R_{12}(t,i)R_{22}(t,i)^{-1}R_{12}^{\top}(t,i) \prec 0, \forall(t,i) \in \mathbb{R}_+ \times \mathfrak{N}.
\end{equation*} It can be directly derived from Remark 3 in \cite{Dragan2020}.    
\end{remark}

\section{The Main Results}
\label{sec:main_results}

\subsection{Algorithm Design}
\label{sec:algorithm_design}

In this paper, we aim to numerically solve for the stabilizing solution of Regime-Switching Stochastic GTREDs with given parameter within a deterministic framework. 

Let $\mathbb{X}^{(h)}(\cdot) = (X^{(h)}(\cdot,1), \ldots, X^{(h)}(\cdot,N))$. We initialize $ X^{(0)}(t,i) = 0$ for all $(t, i) \in \mathbb{R}_+ \times \mathfrak{N}$, and then construct the sequence $\{\mathbb{X}^{(h)}(\cdot)\}_{h\geq0} $, where each component $X^{(h)}(\cdot,i)$ ($1 \leq i \leq N$) is the unique minimal positive semidefinite solution to the following Riccati differential equations:
\begin{equation}
\label{algorithm}
    \begin{aligned}
        &\frac{d}{dt}{X^{(h)}}(t,i) + \hat{A}^{\top}(t,i)X^{(h)}(t,i) + X^{(h)}(t,i)\hat{A}(t,i) + M^{(h)}(t,i) \\
        &- \left( X^{(h)}(t,i)B_0(t,i) + L^{(h)}(t,i) \right)R^{(h)}(t,i)^{-1} \left( B_0^{\top}(t,i)X^{(h)}(t,i) + L^{(h) \top}(t,i) \right)=0 , t \ge 0.    
    \end{aligned}
\end{equation}
The matrices $M^{(h)}(t,i),L^{(h)}(t,i),R^{(h)}(t,i)$ evolve as:  
\[
    \begin{cases}
       M^{(h)}(t,i)=\Pi_1(t)\left[\mathbb{X}^{(h-1)}\right](i)+M(t,i),\forall (t, i) \in \mathbb{R}_+ \times \mathfrak{N};\\
       L^{(h)}(t,i)=\Pi_2(t)\left[\mathbb{X}^{(h-1)}\right](i)+L(t,i),\forall (t, i) \in \mathbb{R}_+ \times \mathfrak{N};\\
       R^{(h)}(t,i)=\Pi_3(t)\left[\mathbb{X}^{(h-1)}\right](i)+R(t,i),\forall (t, i) \in \mathbb{R}_+ \times \mathfrak{N}.
    \end{cases}
\]  

Formally, this aligns with the setup of the zero-sum game in a deterministic framework. The problem formulation of linear-quadratic deterministic differential game is presented below:
\begin{equation} 
\label{pf:lqdg_sde}
    \begin{cases} 
        &dx(t,i) = \left( \hat{A}(t,i)x(t,i) + B_{01}(t,i)u_1(t,i)+B_{02}(t,i)u_2(t,i) \right) dt \\
        &x(0,i) = x_0
    \end{cases}, 
\end{equation}
in which $x_0 \in\mathbb{R}^n$, and Player 1 and Player 2 share the same performance functional:
\begin{equation}
\label{pf:lqdg_pf}
    J_{i}^{(h)}(x_0;u_1,u_2)\triangleq\int_{0}^{\infty}\Bigg[\Bigg\langle\begin{pmatrix}
    M^{(h)}(t,i) & L^{(h)}_1(t,i)  & L^{(h)}_2(t,i) \\
    L^{(h) \top}_1(t,i) & R^{(h)}_{11}(t,i) & R^{(h)}_{12}(t,i)\\
    L^{(h) \top}_2(t,i) & R^{(h) \top}_{12}(t,i) & R^{(h)}_{22}(t,i)
    \end{pmatrix}\begin{pmatrix}
    x_u(t,i)\\
    u_1(t,i)\\
    u_2(t,i)
    \end{pmatrix},\begin{pmatrix}
    x_u(t,i)\\
    u_1(t,i)\\
    u_2(t,i)
    \end{pmatrix}\Bigg\rangle
    \Bigg]dt,
\end{equation}
where $x_{u}(t,i)$ is the solution of \eqref{pf:lqdg_sde} corresponding to input $u_1(t,i)$ and $u_2(t,i)$. The matrix partitioning involved in the above problem description is consistent with that of the corresponding matrix in the Section \ref{subsec:gtrdes}. 
As a special case of ZSLQSDG problem, the core symbol set is expressed as follows:  
\[
\boldsymbol{\Sigma}^{(h)}= (\mathbb{\hat{A}}(\cdot),\mathbb{B}_0(\cdot),\mathbb{Q}^{(h)}(\cdot))
\] and its components as  
$\mathbb{Q}^{(h)}(t) = (Q^{(h)}(t,1), \ldots, Q^{(h)}(t,N)) \in S_{n+m}^N $, with each symmetric component  
\[
    Q^{(h)}(t,i) = \begin{pmatrix} M(t,i) & L(t,i) \\ L^{\top}(t,i) & R(t,i) \end{pmatrix} +\begin{pmatrix} \Pi_1(t)[\mathbb{X}^{(h-1)}](i) & \Pi_2(t)[\mathbb{X}^{(h-1)}](i) \\ \Pi_2(t)[\mathbb{X}^{(h-1)}](i)^{\top} & \Pi_3(t)[\mathbb{X}^{(h-1)}](i) \end{pmatrix}, \forall (t, i) \in \mathbb{R}_+ \times \mathfrak{N}.
\]

\begin{definition}
The system \eqref{pf:lqdg_sde} denoted as 
\[
\left[ \mathbb{A}_0(\cdot) ; \mathbb{B}_{0}(\cdot) \right],
\] 
is called stabilizable if there exists a tuple $\mathbb{\varTheta}(\cdot)= \left(\varTheta(\cdot,1),\ldots, \varTheta(\cdot,N)\right)$ with the $\theta$-periodic and continuous function $\varTheta(\cdot,i)(1 \leq i \leq N) : \mathbb{R}_+ \to \mathbb{R}^{n \times p}$ such that the system $\left( \mathbb{A}_0(\cdot)+ \mathbb{B}_{0}(\cdot) \mathbb{\varTheta}(\cdot) \right)$ is stable.
\end{definition}

\begin{remark}
In the deterministic problem corresponding to system \eqref{pf:lqdg_sde}, the different system modes are practically non-interfering, enabling stabilizability to be defined independently for the deterministic subsystem associated with each mode. For consistency and to facilitate a unified treatment of all modes in subsequent proofs, we adopt definitions analogous to those used in stochastic systems. 
\end{remark}

\begin{definition}[\cite*{Dragan2013book}]
Consider the following stochastic observation system:
\begin{equation}
\label{pr:stochastic_observation_system}
    \begin{cases}
         dx(t) = A_0(t,\eta_t)x(t)dt + \sum_{k=1}^{r}A_k(t,\eta_t)x(t)dw_k(t) \\
         dz(t) = C_0(t,\eta_t)x(t)dt 
    \end{cases}
\end{equation}
where $C_0(t,i) \in \mathbb{R}^{q \times n}$ for all $(t,i) \in \mathbb{R}_{+} \times \mathfrak{N}$.
The system $\left[ \mathbb{C}(\cdot); \mathbb{A}_{0}(\cdot), \dotsb, \mathbb{A}_{r}(\cdot) \right]$ is called \textit{detectable} if there exists a tuple $\mathbb{\varTheta}(\cdot)= \left(\varTheta(\cdot,1),\ldots, \varTheta(\cdot,N)\right)$ with the $\theta$-periodic and continuous function $\varTheta(\cdot,i)(1 \leq i \leq N) : \mathbb{R}_+ \to \mathbb{R}^{n \times p}$ such that the system $(\mathbb{A}_{0}(\cdot)+\mathbb{\varTheta}(\cdot) \mathbb{C}_{0}(\cdot),\dotsb, \mathbb{C}_{r}(\cdot))$ is stable. 
\end{definition}

In the subsequent content, we will show:
\begin{itemize}
    \item The sequences $\{\mathbb{X}^{(h)}(\cdot)\}_{h\geq0} $ are well-defined, and it is the unique minimal positive semidefinite solution of the Riccati differential equation \eqref{algorithm} that satisfy the sign conditions \ref{pf:sign_conditions_1}. 
    \item These sequences are convergent and we have  
    \[
    0 = X^{(0)}(t,i) \preceq \dots \preceq X^{(h)}(t,i) \preceq X^{(h+1)}(t,i) \preceq \dots \preceq \tilde{X}(t,i),\\
    \lim_{h\to\infty} X^{(h)}(t,i) = \tilde{X}(t,i), \forall \, (t, i) \in  \mathbb{R}_{+} \times \mathfrak{N},
    \]
    where $\tilde{\mathbb{X}}(\cdot)$ is the unique $\theta$-periodic and stabilizing solution solution to the Regime-Switching Stochastic GTRDEs \eqref{pf:gtrde}.
\end{itemize}

\begin{remark}
As a special case of ZSLQSDG problem, the core notation set related to Riccati differential equations \eqref{pr:gtrde_kw} as
$\Sigma_{\mathbb{KW}}^{(h)}= (\mathbb{\hat{A}}_{0\mathbb{K}}(\cdot),\mathbb{B}_{0\mathbb{W}}(\cdot),\mathbb{Q}_{\mathbb{KW}}^{(h)}(\cdot))$ and its components as follows:  
\[
\mathbb{Q}_{\mathbb{KW}}^{(h)}(t) = (\mathbb{Q}_{\mathbb{KW}}^{(h)}(t,1), \ldots, \mathbb{Q}_{\mathbb{KW}}^{(h)}(t,N)) \in \mathbb{S}_{n+m}^N, 
\] with each symmetric component  
\[
    \mathbb{Q}_{\mathbb{KW}}^{(h)}(t,i) = \begin{pmatrix} M_{K}(t,i) & L_{KW}(t,i) \\ L_{KW}^{\top}(t,i) & R_{W}(t,i) \end{pmatrix} +\begin{pmatrix} \Pi_{1\mathbb{K}}(t)[\mathbb{X}^{(h-1)}](i) & \Pi_{2\mathbb{KW}}(t)[\mathbb{X}^{(h-1)}](i) \\ \Pi_{2\mathbb{KW}}(t)[\mathbb{X}^{(h-1)}](i)^{\top} & \Pi_{3\mathbb{W}}(t)[\mathbb{X}^{(h-1)}](i) \end{pmatrix}, \forall (t, i) \in \mathbb{R}_+ \times \mathfrak{N}.
\]
\end{remark}

\subsection{Convergence of the Iterative Sequence}

Before introducing the convergence of sequences, we first present the equivalent conditions for the non-emptiness of the set $\mathcal{A}^{\Sigma}$ and some useful properties.  

\begin{proposition}
\label{main:proposition}
Under the Assumption \ref{pf:ass_1}, the following are equivalent:
\begin{enumerate}
    \item[$(i.)$] $ (\mathbb{K}(\cdot), \mathbb{W}(\cdot)) \in \mathcal{A}^\Sigma $.
    \item[$(ii.)$] There exists a function $t \to \mathbb{Y}(t) = (Y(t,1), \ldots, Y(t,N)): \mathbb{R}_+ \to \mathbb{S}_n^N$ which is differentiable with continuous derivative and periodic, with period $\theta$, satisfying the following systems of linear matrix inequalities (LMIs):
    \begin{equation}
    \label{LMIs_kw}
    \begin{pmatrix}
    \frac{d}{dt}{Y}(t,i) +\mathcal{L}_{\mathbb{K}}(t)[\mathbb{Y}](i)+ \Pi_{1\mathbb{K}}(t)[\mathbb{Y}](i)+M_K(t,i) & Y(t,i)B_{0W}(t,i) +\Pi_{2\mathbb{KW}}(t)[\mathbb{Y}](i)+ L_{KW}(t,i) \\
    B^{\top}_{0W}(t,i)Y(t,i) +\Pi_{2\mathbb{KW}}(t)[\mathbb{Y}](i)^{\top}+ L^{\top}_{KW}(t,i) & R_{W}(t,i)+\Pi_{3\mathbb{W}}(t)[\mathbb{Y}](i)
    \end{pmatrix} 
    \end{equation}
    $\prec 0, Y(t,i) \succ 0, \forall (t,i) \in \mathbb{R}_+ \times \mathfrak{N} $,
    where $\mathcal{L}_{\mathbb{K}}(t)[\mathbb{Y}](i)=\hat{A}_{0K}^{\top}(t,i)Y(t,i) + Y(t,i)\hat{A}_{0K}(t,i)$.
    \item[$(iii.)$] The Riccati differential equation \eqref{pr:gtrde_kw} has a unique $\theta$-periodic solution $\tilde{\mathbb{Y}}_{\mathbb{K}\mathbb{W}}(\cdot)$ which has the following properties:
    \begin{enumerate}
        \item[(a.)] $ 0 \preceq \tilde{Y}_{\mathbb{K}\mathbb{W}}(t,i) \preceq \hat{Y}(t,i) $ holds for all $ (t,i) \in \mathbb{R}_+ \times \mathfrak{N} $, where $ \hat{\mathbb{Y}}(t) = (\hat{Y}(t,1), \ldots, \hat{Y}(t,N)) $ is a differentiable function with a continuous derivative, is periodic with period $ \theta $, and satisfies the LMIs \ref{LMIs_kw}.
        \item[(b.)] 
        \begin{equation}
        \label{sign_condition_32}
        R_{W}(t, i) + \sum_{k=1}^r B_{kW}^{\top}(t, i)\tilde{Y}_{\mathbb{K}\mathbb{W}}(t,i)B_{kW}(t, i) \prec 0, \forall (t,i) \in \mathbb{R}_+ \times \mathfrak{N}.
        \end{equation}
        \item[(c.)] $ 0 \preceq \tilde{Y}_{\mathbb{K}\mathbb{W}}(t,i) \preceq Y(t,i), \forall (t,i) \in \mathbb{R}_+ \times \mathfrak{N} $, $\mathbb{Y}(t) = (Y(t,1), \ldots, Y(t,N))$ being an arbitrary global and positive semidefinite solution of the Riccati differential equation \eqref{pr:gtrde_kw} that satisfies the sign condition \eqref{sign_condition_32}.
        \item[(d.)] $\tilde{\mathbb{Y}}_{\mathbb{K}\mathbb{W}}(\cdot)$  is the stabilizing solution. This means the system 
        \[
        \left( \mathbb{A}_0(\cdot) + \mathbb{B}_0(\cdot)\mathbb{L_{\tilde{\mathbb{Y}}_{\mathbb{K}\mathbb{W}}}}(\cdot), \, \dots, \, \mathbb{A}_r(\cdot) + \mathbb{B}_r(\cdot)\mathbb{L_{\tilde{\mathbb{Y}}_{\mathbb{K}\mathbb{W}}}}(\cdot) \right)
        \] is stable, where $\mathbb{L_{\tilde{\mathbb{Y}}_{\mathbb{K}\mathbb{W}}}}(\cdot)$ is defined in \eqref{pr:feedback_gain_l}.
    \end{enumerate}
\end{enumerate}
\end{proposition}

\begin{proof}
$ (i.) \implies (ii.) $  
Since $ (\mathbb{K}(\cdot), \mathbb{W}(\cdot)) \in \mathcal{A}^\Sigma $, the system $ (\mathbb{A}_{0\mathbb{K}}(\cdot),\dots,\mathbb{A}_{r\mathbb{K}}(\cdot)) $ is stable. From (v) of Theorem 3.3.1. in \cite{Dragan2013book}, exists a function $t \to \mathbb{Y}(t) = (Y(t,1), \ldots, Y(t,N)): \mathbb{R}_+ \to \mathbb{S}_n^N$ which is differentiable with continuous derivative and periodic, with period $\theta$, satisfying the following systems of linear matrix inequalities (LMIs):  
\begin{align*}
&\frac{d}{dt}{Y}(t,i) +\mathcal{L}_{\mathbb{K}}(t)[\mathbb{Y}](i) +\Pi_{1\mathbb{K}}(t)[\mathbb{Y}](i)+ M_K(t,i)\\
&- \left( Y(t,i)B_{0W}(t,i) +\Pi_{2\mathbb{KW}}(t)[\mathbb{Y}](i)+ L_{KW}(t,i) \right)\left( R_{W}(t,i)+\Pi_{3\mathbb{W}}(t)[\mathbb{Y}](i) \right)^{-1} \\
&\left( B^{\top}_{0W}(t,i)Y(t,i) +\Pi_{2\mathbb{KW}}(t)[\mathbb{Y}](i)^{\top}+ L^{\top}_{KW}(t,i)\right) \prec 0,\,\text{and}\, Y(t,i) \succ 0,\, t \ge 0, \, 1 \leq i \leq N.
\end{align*}
Moreover, this solution satisfies the sign condition \ref{sign_condition_32}. 
Using the Schur complement, the inequality is reformulated as the LMIs given in \eqref{LMIs_kw}. Thus, $ \mathbb{Y}(\cdot) $ satisfies LMIs \eqref{LMIs_kw}.

$ (ii.) \implies (iii.) $  
Assume there exists a function $ \mathbb{Y}(\cdot) $ as in $ (ii) $. Since $M(t,i) - L_2(t,i)R_{22}(t,i)^{-1}L_2^{\top}(t,i) \succeq 0$ for all $(t,i) \in \mathbb{R}_+ \times \mathfrak{N}$ in Assumption \ref{pf:ass_1}, we have  
\begin{align*}
&\frac{d}{dt}{Y}(t,i) +\mathcal{L}_{\mathbb{K}}(t)[\mathbb{Y}](i) +\Pi_{1\mathbb{K}}(t)[\mathbb{Y}](i) \prec 0, t \ge 0; \, 1 \leq i \leq N.
\end{align*}
This means the system $ (\mathbb{A}_{0\mathbb{K}},\dots,\mathbb{A}_{r\mathbb{K}}) $ stable by using (vi) of Theorem 3.3.1. in \cite{Dragan2013book}. So the conditions of Theorem 5.6.4. in \cite{Dragan2013book} are satisfied. 
Under these conditions the Riccati differential equation \eqref{pr:gtrde_kw} has a bounded solution $\tilde{Y}(t,i)$ which verifies $0 \preceq \tilde{Y}(t,i) \preceq \check{Y}(t,i);\forall (t,i) \in \mathbb{R}_+ \times \mathfrak{N}$ for any bounded solutions $\check{Y}(t,i)$ of the inequality \eqref{pr:gtrde_kw} which satisfies sign condition \ref{sign_condition_32}.
Then $R_{W}(t, i) + \sum_{k=1}^r B_{kW}^{\top}(t, i)\tilde{Y}(t,i)B_{kW}(t, i) \preceq R_{W}(t, i) + \sum_{k=1}^r B_{kW}^{\top}(t, i)\check{Y}(t,i)B_{kW}(t, i) \prec 0$, we get $(a.)$ and $(b.)$ hold.
Let $\mathbb{Y}(t) = (Y(t,1), \ldots, Y(t,N))$ being an arbitrary global and positive semidefinite solution of the Riccati differential equation \eqref{pr:gtrde_kw} that satisfies the sign condition \ref{sign_condition_32}. By transforming the Riccati differential equation \eqref{pr:gtrde_kw}, we obtain $\mathbb{Y}(\cdot)$ satisfied conditions of Theorem 5.6.4. in \cite{Dragan2013book}. This means $\tilde{Y}(t,i) \preceq Y(t,i);\forall (t,i) \in \mathbb{R}_+ \times \mathfrak{N}$, the $(c.)$ is proved.

For every bounded uniform positive and continuous function $\mathbb{H}(t) = (H(t,1), \ldots, H(t,N))$, we has a positive semidefinite solution $\mathbb{Y}(\cdot)=(Y(\cdot,1),\ldots,Y(\cdot,N))$ satisfies the following Riccati differential equation:
\begin{equation*}
\begin{split}    
& \frac{d}{dt}{Y}(t,i) +\mathcal{L}_{\mathbb{K}}(t)[\mathbb{Y}](i)+ \Pi_{1\mathbb{K}}(t)[\mathbb{Y}](i)+M_K(t,i) - \left( Y(t,i)B_{0W}(t,i) +\Pi_{2\mathbb{KW}}(t)[\mathbb{Y}](i)+ L_{KW}(t,i) \right)+H(t,i)\\
& \quad \left( R_{W}(t,i)+\Pi_{3\mathbb{W}}(t)[\mathbb{Y}](i)  \right)^{-1}\left( B^{\top}_{0W}(t,i)Y(t,i) +\Pi_{2\mathbb{KW}}(t)[\mathbb{Y}](i)^{\top}+ L^{\top}_{KW}(t,i) \right) =0, t \ge 0; 1 \leq i \leq N.
\end{split}
\end{equation*}
by using the Theorem 3.3.1. in \cite{Dragan2013book}.

Setting $\Delta(t,i)=Y(t,i)-\tilde{Y}(t,i),\forall (t,i) \in \mathbb{R}_+ \times \mathfrak{N}$ and transforming the Riccati differential equation \eqref{pr:gtrde_kw} using the Proposition 3.3 in \cite{wang2025}, we obtain:
\begin{align*}
& \frac{d}{dt}{Y}(t,i) +\mathcal{L}_{\mathbb{K}}(t)[\mathbb{Y}](i)+ \Pi_{1\mathbb{K}}(t)[\mathbb{Y}](i)+M_K(t,i) - \left( Y(t,i)B_{0W}(t,i) +\Pi_{2\mathbb{KW}}(t)[\mathbb{Y}](i)+ L_{KW}(t,i) \right)\\
& \quad \left( R_{W}(t,i)+\Pi_{3\mathbb{W}}(t)[\mathbb{Y}](i)  \right)^{-1}\left( B^{\top}_{0W}(t,i)Y(t,i) +\Pi_{2\mathbb{KW}}(t)[\mathbb{Y}](i)^{\top}+ L^{\top}_{KW}(t,i) \right) +H(t,i)\\
\end{align*}
\begin{align*}
=& \frac{d}{dt}{Y}(t,i)+\mathcal{L}_{\mathbb{K}}(t)[\mathbb{\tilde{Y}}](i)+ \Pi_{1\mathbb{K}}(t)[\mathbb{\tilde{Y}}](i)+M_K(t,i) - \left( Y(t,i)B_{0W}(t,i) +\Pi_{2\mathbb{KW}}(t)[\mathbb{\tilde{Y}}](i)+ L_{KW}(t,i) \right)\\
& \quad \left( R_{W}(t,i)+\Pi_{3\mathbb{W}}(t)[\mathbb{\tilde{Y}}](i)  \right)^{-1}\left( B^{\top}_{0W}(t,i)Y(t,i) +\Pi_{2\mathbb{KW}}(t)[\mathbb{\tilde{Y}}](i)^{\top}+ L^{\top}_{KW}(t,i) \right) + \sum_{j=1}^N q_{ij}Y(t,j)\\
&+\left( A_{0K}(t,i) + B_{0W}L_{\tilde{Y}}(t,i)\right)^{\top}\Delta(t,i) + \Delta(t,i)\left( A_{0K}(t,i) + B_{0W}L_{\tilde{Y}}(t,i)\right) \\
&+ \sum_{k=1}^r \left( A_{kK}(t,i) + B_{kW}L_{\tilde{Y}}(t,i)\right)^{\top}\Delta(t,i)\left( A_{kK}(t,i) + B_{kW}L_{\tilde{Y}}(t,i)\right)\\
&- \left[ \Delta(t,i)B_{0W}(t,i) + \sum_{k=1}^r \left( A_{kK}(t,i) + B_{kW}L_{\tilde{Y}}(t,i)\right)^{\top}\Delta(t,i)B_{kW}(t,i) \right]\left( R_W(t,i) + \sum_{k=1}^r B_{kW}^{\top}(t,i)Y(t,i)B_{kW}(t,i) \right)^{-1} \\
& \quad \left[ B_{0W}^{\top}(t,i)\Delta(t,i) + \sum_{k=1}^r B_{kW}^{\top}(t,i)\Delta(t,i)\left( A_{kK}(t,i) + B_{kW}L_{\tilde{Y}}(t,i)\right) \right]+H(t,i)\\
=&\frac{d}{dt}{\Delta}(t,i)+\sum_{j=1}^N q_{ij}\Delta(t,j)+\left( A_{0K}(t,i) + B_{0W}L_{\tilde{Y}}(t,i)\right)^{\top}\Delta(t,i) + \Delta(t,i)\left( A_{0K}(t,i) + B_{0W}L_{\tilde{Y}}(t,i)\right)\\
&+ \sum_{k=1}^r \left( A_{kK}(t,i) + B_{kW}L_{\tilde{Y}}(t,i)\right)^{\top}\Delta(t,i)\left( A_{kK}(t,i) + B_{kW}L_{\tilde{Y}}(t,i)\right)\\
&- \left[ \Delta(t,i)B_{0W}(t,i) + \sum_{k=1}^r \left( A_{kK}(t,i) + B_{kW}L_{\tilde{Y}}(t,i)\right)^{\top}\Delta(t,i)B_{kW}(t,i) \right]\left( R_W(t,i) + \sum_{k=1}^r B_{kW}^{\top}(t,i)Y(t,i)B_{kW}(t,i) \right)^{-1} \\
& \quad \left[ B_{0W}^{\top}(t,i)\Delta(t,i) + \sum_{k=1}^r B_{kW}^{\top}(t,i)\Delta(t,i)\left( A_{kK}(t,i) + B_{kW}L_{\tilde{Y}}(t,i)\right) \right]+H(t,i)
\end{align*}    

Since $\Delta(t,i) \succeq 0,\forall (t,i) \in \mathbb{R}_+ \times \mathfrak{N}$, we get the system  
\[
\left( \mathbb{A}_0(\cdot) + \mathbb{B}_0(\cdot)\mathbb{L_{\tilde{\mathbb{Y}}_{\mathbb{K}\mathbb{W}}}}(\cdot), \, \dots, \, \mathbb{A}_r(\cdot) + \mathbb{B}_r(\cdot)\mathbb{L_{\tilde{\mathbb{Y}}_{\mathbb{K}\mathbb{W}}}}(\cdot) \right)
\]
is stable. So $ \tilde{\mathbb{Y}}_{\mathbb{K}\mathbb{W}}(\cdot) $ is stabilizing solution of the Riccati differential equation \eqref{pr:gtrde_kw}. Thus, $ (d.) $ holds.

$(ii.)+(iii.) \implies (i.)$  
If the Riccati differential equation \eqref{pr:gtrde_kw} has a unique $\theta$-periodic stabilizing solution $\tilde{\mathbb{Y}}_{\mathbb{K}\mathbb{W}}(\cdot)$ with the properties in $(iii)$, we immediately having the system $ (\mathbb{A}_{0\mathbb{K}}(\cdot),\dots,\mathbb{A}_{r\mathbb{K}}(\cdot)) $ and the system \[
\left( \mathbb{A}_0(\cdot) + \mathbb{B}_0(\cdot)\mathbb{L_{\tilde{\mathbb{Y}}_{\mathbb{K}\mathbb{W}}}}(\cdot), \, \dots, \, \mathbb{A}_r(\cdot) + \mathbb{B}_r(\cdot)\mathbb{L_{\tilde{\mathbb{Y}}_{\mathbb{K}\mathbb{W}}}}(\cdot) \right)
\] is stable. Then by definition, existing $(\mathbb{K}(\cdot), \mathbb{W}(\cdot)) \in \mathcal{A}^\Sigma$.
\end{proof}

\begin{remark}
\label{main:proposition_kw}
If the Assumption \ref{pf:ass_1} is satisfied and the Regime-Switching Stochastic GTRDEs \eqref{pf:gtrde} has the the unique stabilizing and $\theta$-periodic solution $\mathbb{\tilde{X}}(\cdot) = \bigl(\tilde{X}(\cdot,1), \ldots, \tilde{X}(\cdot,N)\bigr)\in \mathbb{S}^N_n$, then the set $\mathcal{A}^\Sigma$ is non-empty.
Let $\mathbb{K}(\cdot) = \bigl(K(\cdot,1), \ldots, K(\cdot,N)\bigr),\mathbb{W}(\cdot) = \bigl(W(\cdot,1), \ldots, W(\cdot,N)\bigr)$with 
\[K(t,i)=-(R_{22}(t,i)+\sum_{k=1}^{r}B^{\top}_{k2}(t,i)\tilde{X}(t,i)B_{k2}(t,i) )^{-1}(B_{02}^{\top}(t,i)\tilde{X}(t,i) + \sum_{k=1}^{r}B_{k2}^{\top}(t,i)\tilde{X}(t,i)A_k(t,i)),\,\text{and}\] 
\[W(t,i)=-(R_{22}(t,i)+\sum_{k=1}^{r}B^{\top}_{k2}(t,i)\tilde{X}(t,i)B_{k2}(t,i))^{-1}(R_{21}(t,i)+\sum_{k=1}^{r}B^{\top}_{k2}(t,i)\tilde{X}(t,i)B_{k1}(t,i)),\forall (t,i) \in \mathbb{R}_+ \times \mathfrak{N}\]
we have $(\mathbb{K}(\cdot), \mathbb{W}(\cdot)) \in \mathcal{A}^\Sigma$. This result can be derived by applying Proposition 3.2 in \cite{wang2025}, Remark 4.1.5. in \cite{Dragan2013book} and Proposition \ref{main:proposition}, with the detailed proofs omitted.
\end{remark}

\begin{theorem}
Assume the following conditions hold:
\begin{itemize}
    \item Assumption \ref{pf:ass_1} is satisfied.
    \item The set $\mathcal{A}^\Sigma$ is non-empty.
    \item Let $\mathbb{A}^{(0)}_{k}(\cdot) = \left(A^{(0)}_{k}(\cdot,1), \ldots, A^{(0)}_{k}(\cdot,N)\right)$ for $k=0,\ldots,r$. 
    Define:
    \[
    A^{(0)}_{0}(t,i) 
    \;=\; \hat{A}(t,i) - B_{02}(t,i)R_{22}(t,i)^{-1}L^{\top}_{2}(t,i),
    A^{(0)}_{k}(t,i) 
    \;=\; A_{k}(t,i) - B_{k2}(t,i)R_{22}(t,i)^{-1}L^{\top}_{2}(t,i)
    \]
    for all $(t,i) \in \mathbb{R}_+ \times \mathfrak{N}$ and $k=1,\ldots,r$. 
    There exists a tuple 
    $\mathbb{C}^{(0)}(\cdot)= \left(C^{(0)}(\cdot,1),\ldots, C^{(0)}(\cdot,N)\right)$ satisfying
    \[
    C^{(0) \top}(t,i)C^{(0)}(t,i) 
    \;=\; M^{}(t,i) - L^{}_2(t,i)R^{}_{22}(t,i)^{-1}L^{ \top}_2(t,i);\,\, \forall (t,i) \in \mathbb{R}_+ \times \mathfrak{N}
    \]
    such that the system 
    \[
    \bigl[ \mathbb{C}^{(0)}(\cdot); \mathbb{A}^{(0)}_{0}(\cdot), \ldots, \mathbb{A}^{(0)}_{r}(\cdot) \bigr]
    \]
    is \textit{stochastically detectable}.
\end{itemize}
Then we have,
\item[$1.$] For each $h \geq 1$, $X^{(h)}(\cdot,i)(1 \leq i \leq N)$ is well-defined as the unique minimal positive semidefinite solution of the Riccati differential equation \eqref{algorithm} that satisfy the sign conditions \ref{pf:sign_conditions_1}. Moreover, the following properties hold:
\begin{itemize}
    \item[$a_h$.] $X^{(h)}(\cdot,i)(1 \leq i \leq N)$ is a periodic function.
    \item[$b_h$.] Let $\mathbb{A}^{(h)}(\cdot) = \left(A^{(h)}(\cdot,1), \ldots, A^{(h)}(\cdot,N)\right)$ with $A^{(h)}(t,i) \;=\; \hat{A}(t,i) - B_{02}(t,i)R^{(h)}_{22}(t,i)^{-1}L^{(h) \top}_{2}(t,i)$ for all $(t,i) \in \mathbb{R}_+ \times \mathfrak{N}$. If exists a tuple $\mathbb{C}^{(h)}(\cdot)= \left(C^{(h)}(\cdot,1),\ldots, C^{(h)}(\cdot,N)\right)$ satisfying
    \[
    C^{(h) \top}(t,i)C^{(h)}(t,i) 
    \;=\; M^{(h)}(t,i) - L^{(h) }_2(t,i)R^{(h)}_{22}(t,i)^{-1}L^{(h) \top}_2(t,i);\,\, \forall (t,i) \in \mathbb{R}_+ \times \mathfrak{N}
    \]
    such that the deterministic system $\bigl[ \mathbb{C}^{(h)}(\cdot); \mathbb{A}^{(h)}(\cdot) \bigr]$ is detectable
    then $X^{(h)}(\cdot,i)(1 \leq i \leq N)$ coincides with the stabilizing solution of the Riccati differential equation \eqref{algorithm}.
    \item[$c_h$.] For all $(t,i) \in \mathbb{R}_+ \times \mathfrak{N}$, 
          \[
          0 \;=\; X^{(0)}(t,i) \; \preceq \; X^{(1)}(t,i) \; \preceq \; \cdots \; \preceq \; X^{(h-1)}(t,i) \; \preceq \; X^{(h)}(t,i) \; \preceq \; \cdots \; \preceq \; \tilde{X}(t,i),
          \]
          where $\mathbb{\tilde{X}}(t) = \bigl(\tilde{X}(t,1), \ldots, \tilde{X}(t,N)\bigr)\in \mathbb{S}^N_n$ is the unique stabilizing and $\theta$-periodic solution of Regime-Switching Stochastic GTRDEs \eqref{pf:gtrde}.
\end{itemize}   
\item[$2.$] For all $(t,i) \in \mathbb{R}_+ \times \mathfrak{N}$, 
          \[
          \lim_{h \to \infty} X^{(h)}(t,i) \;=\; \tilde{X}(t,i).
          \]
\end{theorem}

\begin{proof}
We use induction on $h = 1, 2, \ldots$ to establish that the Riccati differential equation \eqref{algorithm} are mathematically well-posed and possess minimal, positive semidefinite solutions satisfying conditions ($a_h$)-($c_h$).    

For $h = 1$, the Riccati differential equations \eqref{algorithm} reduce to
\begin{equation}
\label{proof_1}
\begin{split}
\frac{d}{dt}X(t,i) & + \hat{A}^{\top}(t,i)X(t,i) + X(t,i)\hat{A}(t,i) + M(t,i) \\
& - ( X(t,i)B_0(t,i) + L(t,i) ) R(t,i)^{-1} ( B_0^{\top}(t,i)X(t,i) + L^{\top}(t,i) ) = 0,t \geq 0; 1 \leq i \leq N.
\end{split} 
\end{equation}
Since the condition  $R_{22}(t,i) \succ 0; \forall (t,i) \in \mathbb{R}_+ \times \mathfrak{N} $ specified in Assumption \ref{pf:ass_1} and $\mathcal{A}^\Sigma$ is not empty we deduce that
\begin{equation*}
R_{11}(t,i) - R_{12}(t,i)R_{22}(t,i)^{-1}R_{12}^{\top}(t,i) \prec 0, \forall(t,i) \in \mathbb{R}_+ \times \mathfrak{N}.
\end{equation*}
This means the invertibility of the matrices $R(t,i)$ is guaranteed and  that signature of $R(t,i)$ matches that of the matrix $\mathrm{diag}(-\mathbb{I}_{m_1}, \mathbb{I}_{m_2})$. It can be found that the Riccati differential equations \eqref{proof_1} is a GTRDEs is related to the zero-sum game problem in a deterministic framework.

Let $(\mathbb{K}(\cdot), \mathbb{W}(\cdot)) \in \mathcal{A}^\Sigma$, we have $Y(t,i) \succ 0; \forall(t,i) \in \mathbb{R}_+ \times \mathfrak{N}$ satisfies the LMIs:
\begin{equation}
\label{proof_2}
\begin{aligned}
&\begin{pmatrix}
\frac{d}{dt}Y(t,i) + \hat{A}_{0K}^{\top}(t,i)Y(t,i) + Y(t,i)\hat{A}_{0K}(t,i) + M_{K}(t,i) & Y(t,i)B_{0W}(t,i) + L_{KW}(t,i) \\
B_{0W}^{\top}(t,i)Y(t,i)+ L_{KW}^{\top}(t,i) & R_{W}(t,i)
\end{pmatrix} \\
&+ \begin{pmatrix} \Pi_{1\mathbb{K}}(t)[\mathbb{Y}](i) & \Pi_{2\mathbb{KW}}(t)[\mathbb{Y}](i) \\ \Pi_{2\mathbb{KW}}(t)[\mathbb{Y}](i)^{\top} & \Pi_{3\mathbb{W}}(t)[\mathbb{Y}](i) \end{pmatrix} \prec 0, \forall(t,i) \in \mathbb{R}_+ \times \mathfrak{N}.
\end{aligned}
\end{equation}
From Remark \ref{pr:pi_kw}, we have
\begin{align*}
     \begin{pmatrix} \Pi_{1\mathbb{K}}(t)[\mathbb{Y}](i) & \Pi_{2\mathbb{KW}}(t)[\mathbb{Y}](i) \\ \Pi_{2\mathbb{KW}}(t)[\mathbb{Y}](i)^{\top} & \Pi_{3\mathbb{W}}(t)[\mathbb{Y}](i) \end{pmatrix} \succeq 0, \forall(t,i) \in \mathbb{R}_+ \times \mathfrak{N},
\end{align*}
then we obtain 
\begin{equation}
\label{proof_3}
\begin{aligned}
    \begin{pmatrix}
    \frac{d}{dt}Y(t,i) + \hat{A}_{0K}^{\top}(t,i)Y(t,i) + Y(t,i)\hat{A}_{0K}(t,i) + M_{K}(t,i) & Y(t,i)B_{0W}(t,i) + L_{KW}(t,i) \\
    B_{0W}^{\top}(t,i)Y(t,i)+ L_{KW}^{\top}(t,i) & R_{W}(t,i)
    \end{pmatrix}  \prec 0,\forall(t,i) \in \mathbb{R}_+ \times \mathfrak{N}.
\end{aligned}    
\end{equation}

The above LMIs \eqref{proof_3} is associated with $\boldsymbol{\Sigma}^{(1)}= (\mathbb{\hat{A}}(\cdot),\mathbb{B}_0(\cdot),\mathbb{Q}^{(1)}(\cdot))$. This implies that $\mathcal{A}^{\boldsymbol{\Sigma}^{(1)}}$ is non-empty if $\mathcal{A}^{\Sigma}$ is non-empty. Therefore, based on Corollary 3 in \cite{Dragan2020}, we deduce that if $X^{(1)}_{\tau}(\cdot,i)(1 \leq i \leq N)$ is the solution to the GTRDEs \eqref{proof_1} satisfying the condition $X^{(1)}_{\tau}(\tau,i) = 0$, is well-defined on $[0, \tau]$ for all $\tau > 0$.
From the Theorem 2. in \cite{Dragan2020} to the special case of the Stochastic GTRDEs, we further conclude that the function $X^{(1)}(\cdot,i)(1 \leq i \leq N)$ defined as:
\begin{equation}
X^{(1)}(t,i) = \lim_{\tau \to \infty} X^{(1)}_{\tau}(t,i), \forall \,t \in \mathbb{R}_+ ,
\end{equation}
is the unique minimal solution among all positive semidefinite solutions to the GTRDEs \eqref{proof_1} that satisfy the sign conditions \ref{pf:sign_conditions_1}. Additionally, since the parameters involved in the GTRDEs \eqref{proof_1} are periodic, the mapping $t \mapsto X^{(h)}(t,i)$ is a periodic function with period $\theta$.
Based on the developments of Chapter 4 in \cite{Dragan2013book} one may deduce that the system 
$
\left[ \mathbb{C}^{(0)}(\cdot); \mathbb{A}^{(0)}_{0}(\cdot) \right]
$
is detectable in the deterministic sense if the system 
$
\left[ \mathbb{C}^{(0)}(\cdot); \mathbb{A}^{(0)}_{0}(\cdot), \ldots, \mathbb{A}^{(0)}_{r}(\cdot) \right]
$
is stochastically detectable. By the Theorem 3. in \cite{Dragan2020}, $X^{(1)}(\cdot,i)$ emerges as the unique bounded and stabilizing solution to the GTRDEs \eqref{algorithm} in the special case where $h = 1$.

The assumptions in this theorem satisfy the assumptions of Theorem 3 in \cite{Dragan2020}, so we can conclude that the Regime-Switching Stochastic GTRDEs \eqref{pf:gtrde} has the unique stabilizing and $\theta$-periodic solution of Regime-Switching Stochastic GTRDEs \eqref{pf:gtrde} that satisfy the sign conditions \ref{pf:sign_conditions_1}. 
Observe that the Regime-Switching Stochastic GTRDEs \eqref{pf:gtrde} satisfied by its stabilizing solution $\tilde{\mathbb{X}}(\cdot)$ may be rewritten as  
\begin{equation}
\label{proof_4}
\begin{aligned}
&\frac{d}{dt}\tilde{X}(t,i) + \hat{A}^{\top}(t,i)\tilde{X}(t,i) + \tilde{X}(t,i)\hat{A}(t,i) + M(t,i)+\Pi_1(t)\left[\tilde{\mathbb{X}}\right](i)\\
&- \left( \tilde{X}(t,i)B_0(t,i) +\Pi_2(t)\left[\tilde{\mathbb{X}}\right](i)+ L(t,i) \right) \left(R(t,i)+\Pi_3(t)\left[\tilde{\mathbb{X}}\right](i) \right) ^{-1}\\
& \times \left( B_0^{\top}(t,i)\tilde{X}(t,i) +\Pi_2(t)\left[\tilde{\mathbb{X}}\right](i)^{\top}+ L(t,i)^{\top} \right) =0,t \geq 0; 1 \leq i \leq N.     
\end{aligned}
\end{equation}
Since $\tilde{X}(t,i) \succeq 0; \forall(t,i) \in \mathbb{R}_+ \times \mathfrak{N}$ we deduce from Remark \ref{pr:pi_kw} that
\begin{align*}
     \begin{pmatrix} \Pi_1(t)[\mathbb{\tilde{X}}](i) & \Pi_2(t)[\mathbb{\tilde{X}}](i) \\ \Pi_2(t)[\mathbb{\tilde{X}}](i)^{\top} & \Pi_3(t)[\mathbb{\tilde{X}}](i) \end{pmatrix} \succeq 0, \forall(t,i) \in \mathbb{R}_+ \times \mathfrak{N},
\end{align*}
then
\[
    \mathbb{Q}(t)[\mathbb{\tilde{X}}](i)=\mathbb{Q}(t)+\begin{pmatrix} \Pi_1(t)[\mathbb{\tilde{X}}](i) & \Pi_2(t)[\mathbb{\tilde{X}}](i) \\ \Pi_2(t)[\mathbb{\tilde{X}}](i)^{\top} & \Pi_3(t)[\mathbb{\tilde{X}}](i) \end{pmatrix} \succeq \mathbb{Q}(t)[\mathbb{X}^{(0)}](i) = \begin{pmatrix} M(t,i) & L(t,i) \\ L^{\top}(t,i) & R(t,i) \end{pmatrix}.
\]
By applying Lemma \ref{comparison_theorem} to the special case of the Stochastic GTRDEs \eqref{proof_1} and \eqref{proof_4} we may infer that $ 0 \preceq X^{(1)}_{\tau}(t,i) \preceq \tilde{X}(t,i) $ for all $t \in [0, \tau]$ where $\tau > 0$ and $1 \leq i \leq N$. Taking the limit for $\tau \to \infty$, we obtain $0 \preceq X^{(1)}(t,i) \preceq X(t,i);\forall (t,i) \in  \mathbb{R}_+ \times \mathfrak{N}$.

Consequently, the claims $(a_1.)-(c_1.)$ specified in the statement are satisfied. 

For $h=l(l\geq 2)$,we assume the functions $X^{(l)}(\cdot,i)(1 \leq i \leq N)$ are well defined as 
the unique minimal solution among all positive semidefinite solutions to the GTRDEs \eqref{algorithm} that satisfy the sign conditions \ref{pf:sign_conditions_1}, and the properties $(a_l.)-(c_l.)$ from the statement hold. We show now that for $h = l+1$, the GTRDEs \eqref{algorithm} admit the unique minimal solution $X^{(l+1)}(\cdot,i)(1 \leq i \leq N)$ among all positive semidefinite solutions to the GTRDEs \eqref{algorithm} that satisfy the sign conditions \ref{pf:sign_conditions_1}. Moreover,
the solutions is $\theta$-periodic function and the properties $(a_{l+1}.)-(c_{l+1}.)$ hold.

For $h = l+1$, the GTRDEs \eqref{algorithm} reduce to
\begin{equation}
\label{proof_5}
\begin{split}
&\frac{d}{dt}X(t,i)  + \hat{A}^{\top}(t,i)X(t,i) + X(t,i)\hat{A}(t,i) + M^{(l+1)}(t,i) \\
& - \left( X(t,i)B_0(t,i) + L^{(l+1)}(t,i) \right) R^{(l+1)}(t,i)^{-1} \left( B_0^{\top}(t,i)X(t,i) + L^{(l+1) \top}(t,i) \right) = 0,t \geq 0; 1 \leq i \leq N.
\end{split} 
\end{equation}
From Remark \ref{main:proposition_kw}, let $\mathbb{K}(t) = \bigl(K(t,1), \ldots, K(t,N)\bigr),\mathbb{W}(t) = \bigl(W(t,1), \ldots, W(t,N)\bigr)$ for all $t \in \mathbb{R}_+$ with 
\[K(t,i)=-(R_{22}(t,i)+\sum_{k=1}^{r}B^{\top}_{k2}(t,i)\tilde{X}(t,i)B_{k2}(t,i) )^{-1}(B_{02}^{\top}(t,i)\tilde{X}(t,i) + \sum_{k=1}^{r}B_{k2}^{\top}(t,i)\tilde{X}(t,i)A_k(t,i)),\,\text{and}\] 
\[W(t,i)=-(R_{22}(t,i)+\sum_{k=1}^{r}B^{\top}_{k2}(t,i)\tilde{X}(t,i)B_{k2}(t,i))^{-1}(R_{21}(t,i)+\sum_{k=1}^{r}B^{\top}_{k2}(t,i)\tilde{X}(t,i)B_{k1}(t,i)),\]
we have $(\mathbb{K}(\cdot), \mathbb{W}(\cdot)) \in \mathcal{A}^\Sigma$.
By using Proposition \ref{main:proposition}, we have $Y(t,i) \succeq \tilde{X}(t,i)$ and $Y(t,i) \succ 0; \forall(t,i) \in \mathbb{R}_+ \times \mathfrak{N}$ satisfies the LMIs:
\begin{equation}
\label{proof_6}
\begin{aligned}
&\begin{pmatrix}
\frac{d}{dt}Y(t,i) + \hat{A}_{0K}^{\top}(t,i)Y(t,i) + Y(t,i)\hat{A}_{0K}(t,i) + M^{(l+1)}_{K}(t,i) & Y(t,i)B_{0W}(t,i) + L^{(l+1)}_{KW}(t,i) \\
B_{0W}^{\top}(t,i)Y(t,i)+ L^{(l+1) \top}_{KW}(t,i) & R^{(l+1)}_{W}(t,i)
\end{pmatrix} \\
&+ \begin{pmatrix} \Pi_{1\mathbb{K}}(t)[\mathbb{Y}-\mathbb{X}^{(l)}](i) & \Pi_{2\mathbb{KW}}(t)[\mathbb{Y}-\mathbb{X}^{(l)}](i) \\ \Pi_{2\mathbb{KW}}(t)[\mathbb{Y}-\mathbb{X}^{(l)}](i)^{\top} & \Pi_{3\mathbb{W}}(t)[\mathbb{Y}-\mathbb{X}^{(l)}](i) \end{pmatrix} \prec 0, \forall(t,i) \in \mathbb{R}_+ \times \mathfrak{N}.
\end{aligned}
\end{equation}
From $(c_l.)$ and Remark \ref{pr:pi_kw}, we have $Y(t,i) \succeq \tilde{X}(t,i) \succeq X^{(l)}(t,i);\forall (t,i) \in  \mathbb{R}_+ \times \mathfrak{N}$ and
\begin{align*}
     \begin{pmatrix} \Pi_{1\mathbb{K}}(t)[\mathbb{Y}-\mathbb{X}^{(l)}](i) & \Pi_{2\mathbb{KW}}(t)[\mathbb{Y}-\mathbb{X}^{(l)}](i) \\ \Pi_{2\mathbb{KW}}(t)[\mathbb{Y}-\mathbb{X}^{(l)}](i)^{\top} & \Pi_{3\mathbb{W}}(t)[\mathbb{Y}-\mathbb{X}^{(l)}](i) \end{pmatrix} \succeq 0, \forall(t,i) \in \mathbb{R}_+ \times \mathfrak{N},
\end{align*}
Then we obtain
\begin{equation}
\label{proof_7}
\begin{aligned}
    \begin{pmatrix}
    \frac{d}{dt}Y(t,i) + \hat{A}_{0K}^{\top}(t,i)Y(t,i) + Y(t,i)\hat{A}_{0K}(t,i) + M^{(l+1)}_{K}(t,i) & Y(t,i)B_{0W}(t,i) + L^{(l+1)}_{KW}(t,i) \\
    B_{0W}^{\top}(t,i)Y(t,i)+ L^{(l+1) \top}_{KW}(t,i) & R^{(l+1)}_{W}(t,i)
    \end{pmatrix}  \prec 0
\end{aligned}    
\end{equation}
for all $(t,i) \in \mathbb{R}_+ \times \mathfrak{N}$.
The above LMIs \eqref{proof_7} is associated with $\boldsymbol{\Sigma}^{(l+1)}= (\mathbb{\hat{A}}(\cdot),\mathbb{B}_0(\cdot),\mathbb{Q}^{(l+1)}(\cdot))$. This implies that $\mathcal{A}^{\boldsymbol{\Sigma}^{(l+1)}}$ is non-empty if $\mathcal{A}^{\Sigma}$ is non-empty. Therefore, based on Corollary 3 in \cite{Dragan2020}, we deduce that if $X^{(l+1)}_{\tau}(\cdot,i)(1 \leq i \leq N)$ is the solution to GTRDEs \eqref{proof_5} satisfying the condition $X^{(l+1)}_{\tau}(\tau,i) = 0$, is well-defined on $[0, \tau]$ for all $\tau > 0$.
By applying Theorem 2. in \cite{Dragan2020} to the special case of the Stochastic GTRDEs, we further conclude that the function $X^{(l+1)}(\cdot,i)$ defined as:
\begin{equation}
X^{(l+1)}(t,i) = \lim_{\tau \to \infty} X^{(l+1)}_{\tau}(t,i), \forall(t,i) \in \mathbb{R}_+ \times \mathfrak{N},
\end{equation}
is the unique minimal solution among all positive semidefinite solutions to the GTRDEs \eqref{proof_5} that satisfy the sign conditions \ref{pf:sign_conditions_1}. Additionally, the mapping $t \mapsto X^{(l+1)}(t,i)$ is a periodic function with period $\theta$.
As a deterministic special case within the stochastic framework, we apply Theorem 3. from \cite{Dragan2020}. Under the specified assumption for $(b_{l+1}.) $, it follows that $X^{(l+1)}(\cdot,i)$ emerges as the unique bounded and stabilizing solution to the GTRDEs \eqref{algorithm} in the special case where $h = l+1$.

Since $\tilde{X}(t,i)-X^{(l)}(t,i) \succeq 0,X^{(l)}(t,i)-X^{(l-1)}(t,i) \succeq 0; \forall(t,i) \in \mathbb{R}_+ \times \mathfrak{N}$ in $(c_l.)$ and Remark \ref{pr:pi_kw}, we have
\begin{align*}
     \begin{pmatrix} \Pi_1(t)[\mathbb{\tilde{X}}-\mathbb{X}^{(l)}](i) & \Pi_2(t)[\mathbb{\tilde{X}}-\mathbb{X}^{(l)}](i) \\ \Pi_2(t)[\mathbb{\tilde{X}}-\mathbb{X}^{(l)}](i)^{\top} & \Pi_3(t)[\mathbb{\tilde{X}}-\mathbb{X}^{(l)}](i) \end{pmatrix} \succeq 0,\forall(t,i) \in \mathbb{R}_+ \times \mathfrak{N},
\end{align*}
and 
\begin{align*}
     \begin{pmatrix} \Pi_1(t)[\mathbb{\mathbb{X}}^{(l)}-\mathbb{X}^{(l-1)}](i) & \Pi_2(t)[\mathbb{\mathbb{X}}^{(l)}-\mathbb{X}^{(l-1)}](i) \\ \Pi_2(t)[\mathbb{\mathbb{X}}^{(l)}-\mathbb{X}^{(l-1)}](i)^{\top} & \Pi_3(t)[\mathbb{\mathbb{X}}^{(l)}-\mathbb{X}^{(l-1)}](i) \end{pmatrix} \succeq 0, \forall(t,i) \in \mathbb{R}_+ \times \mathfrak{N}.
\end{align*}
This means
\[
    \mathbb{Q}(t)[\mathbb{\tilde{X}}](i) \succeq \mathbb{Q}(t)[\mathbb{X}^{(l)}](i)= \mathbb{Q}^{(l+1)}(t,i);\mathbb{Q}^{(l+1)}(t,i)= \mathbb{Q}(t)[\mathbb{X}^{(l)}](i) \succeq \mathbb{Q}(t)[\mathbb{X}^{(l-1)}](i)=\mathbb{Q}^{(l)}(t,i) .
\]
By applying Lemma \ref{comparison_theorem} in the special case of Stochastic GTRDEs \eqref{proof_4},\eqref{proof_5} and \eqref{algorithm} with $h=l$, we may infer that $X^{{(l+1)}}_{\tau}(t,i) \preceq \tilde{X}(t,i) $ and $X^{{(l)}}_{\tau}(t,i) \preceq X^{{(l+1)}}_{\tau}(t,i) $ for all $t \in [0, \tau]$ where $\tau > 0$ and $1 \leq i \leq N$. Taking the limit for $\tau \to \infty$, we obtain $0 \preceq X^{(l)}(t,i) \preceq X^{(l+1)}(t,i) \preceq \tilde{X}(t,i);\forall (t,i) \in \mathbb{R}_+ \times \mathfrak{N}$.

Consequently, the claims $(a_{l+1}.)-(c_{l+1}.)$ specified in the statement are satisfied. 

Finally, since the sequence $\{\mathbb{X}^{(h)}(\cdot)\}_{h\geq0}$ is monotonic and bounded above by $\tilde{\mathbb{X}}(\cdot)$, it converges pointwise to $\mathbb{X}^{(*)}(\cdot)$, where $X^{(*)}(t,i) = \lim_{h \to \infty} X^{(h)}(t,i) \preceq \tilde{X}(t,i)$ for all $(t,i) \in \mathbb{R}_+ \times \mathfrak{N}$. Using standard arguments based on Lebesgue's dominated convergence theorem, one can show that the mapping \[ t \mapsto \mathbb{X}^{(*)}(t) = \left(X^{(*)}(t,1), \ldots, X^{(*)}(t,N)\right) \] constitutes a global solution to the Regime-Switching Stochastic GTRDEs \eqref{pf:gtrde} and satisfies the sign conditions \ref{pf:sign_conditions_1}.
By Theorem 2. and Theorem 3. in \cite{Dragan2020}, we know that $\tilde{X}(t,i) \preceq X^{(*)}(t,i)$ for all $(t,i) \in \mathbb{R}_+ \times \mathfrak{N}$. This allows us to conclude that $X^{(*)}(t,i) = \tilde{X}(t,i)$ for all $(t,i) \in \mathbb{R}_+ \times \mathfrak{N}$, thereby completing the proof.
\end{proof}

\section{Numerical Experiments}
\label{sec:numerical_experiments}

Given that regime-switching and periodic-coefficient problems can be decomposed into deterministic subproblems, we perform extensive numerical experiments with randomly generated system parameters across various system modes and dimensions to validate the effectiveness and robustness of the proposed algorithm. Key parameter settings are as follows:

\begin{table}[htbp]
\centering
\caption{Summary of experimental parameters}
\label{tab:param_summary}
\begin{tabular}{@{}p{6cm}p{9cm}@{}}
\toprule
Parameter ($i=1,2$) & Setting \\
\midrule
Convergence tolerance & $1 \times 10^{-12}$ \\
Range of system dimensions ($n$) & From 1 to 20 \\
Trials per dimension & Fixed at 1,000 \\
Total number of trials & 20,000 (20 dimensions $\times$ 1,000 trials) \\
$n \times n$ matrices $A_0(i)$, $A_1(i)$, $A_2(i)$ & Randomly generated with elements following the standard normal distribution\\
$n \times n$ matrices $B_{01}(i)$, $B_{02}(i)$ & $B_{01}(i) = 4I_{n} - 0.5H_1(i)$ and $B_{02}(i) = 7I_{n} + 0.5H_2(i)$, where $H_1(i)$ and $H_2(i)$ are random matrices with entries uniformly distributed over $[0, 1]$ \\
$n \times n$ matrices $B_{11}(i)$, $B_{12}(i)$, $B_{21}(i)$, $B_{22}(i)$ & Randomly generated with elements following the uniformly distributed over $[0, 0.01]$\\
$n \times n$ matrix $R_{11}(i)$ & $-7I_{n} - U_{11}^{\top}(i)U_{11}(i)$, where $U_{11}(i)$ is a random matrix with entries uniformly distributed over $[0, 1]$ \\
$n \times n$ matrix  $R_{22}(i)$ & $4I_{n} + U_{22}^{\top}(i)U_{22}(i)$, where $U_{22}(i)$ is a random matrix with entries uniformly distributed over $[0, 1]$ \\
Matrices $R_{12}(i),R_{21}(i), L_1(i)$, $L_2(i)$ & Randomly generated with elements following the uniform distribution over $[0, 0.1]$ \\
Matrix $M(i)$ & Calculated as $M(i) = U^\top(i) U(i) + 0.1I_{n \times n}$, where $U(i)$ is random matrices with elements following the standard normal distribution\\
$Q = (q_{ij})_{2 \times 2}$ & The generator matrix is generated by creating $1$ random non-negative numbers as transition rates from each state to other states, with the diagonal elements set to the negative of the sum of the non-diagonal elements in the corresponding row.\\
\bottomrule
\end{tabular}
\end{table}

\begin{figure}[htbp]
    \centering
    \includegraphics[width=0.9\textwidth]{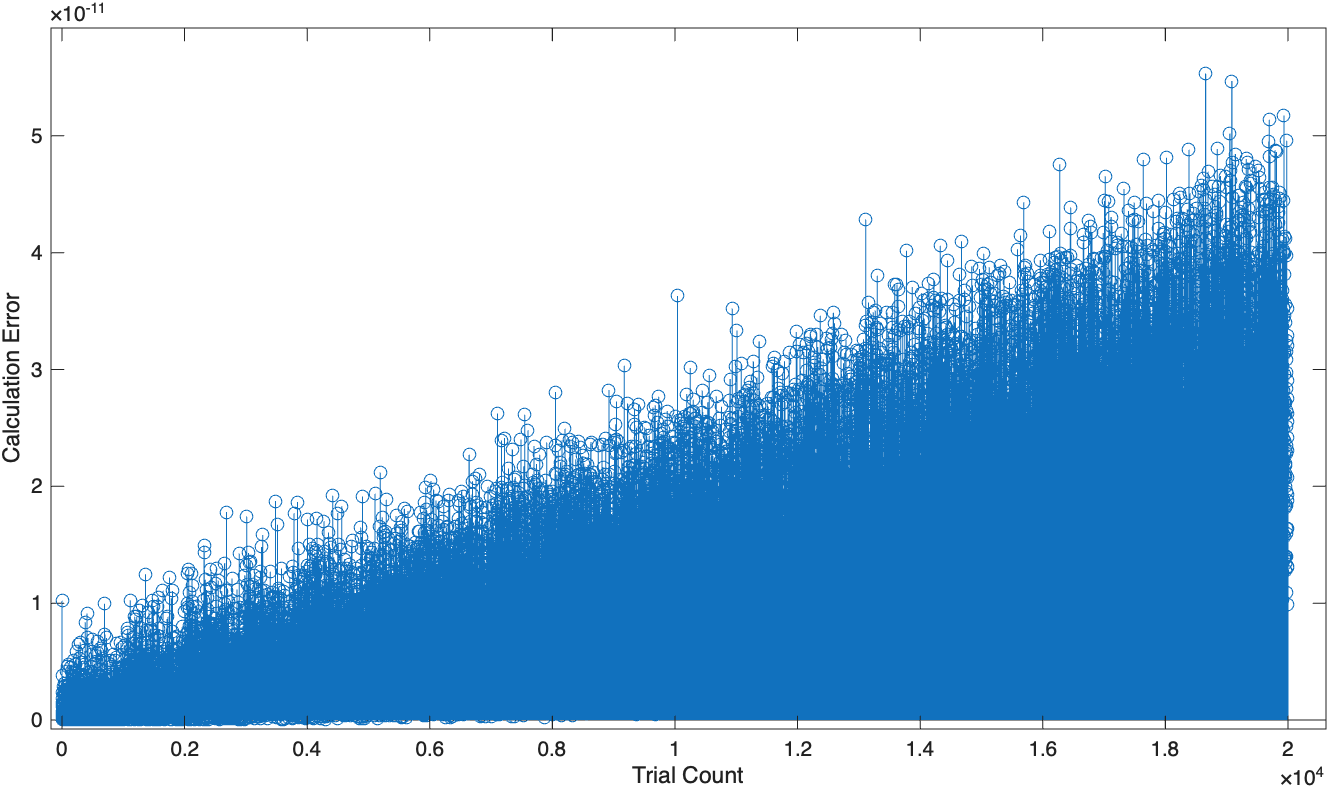} 
    \caption{The error of the Stochastic GTRDEs corresponding to System Mode 1}
    \label{fig:Mode_1}
\end{figure}
\begin{figure}[htbp]
    \centering
    \includegraphics[width=0.9\textwidth]{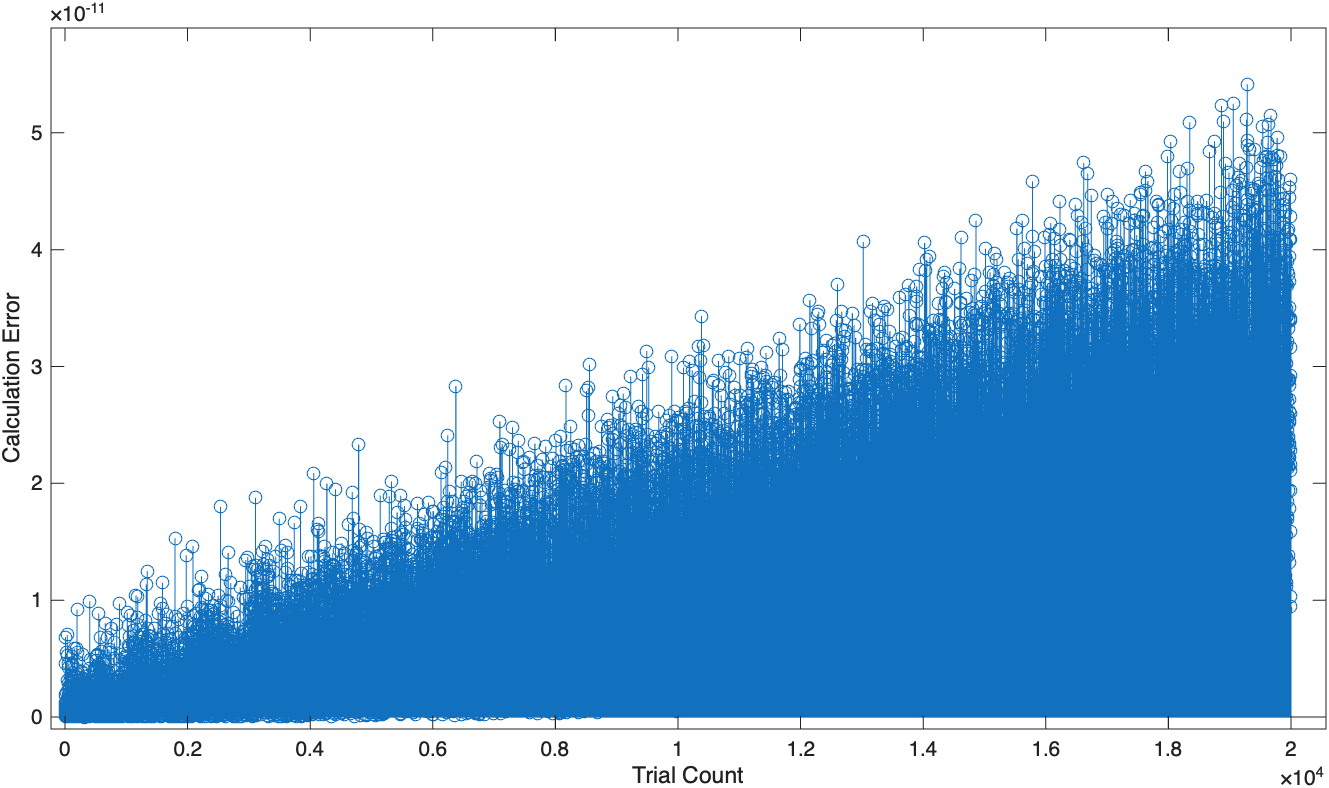} 
    \caption{The error of the Stochastic GTRDEs corresponding to System Mode 2}
    \label{fig:Mode_2}
\end{figure}

Under the default random seed setting in MATLAB, we conducted 1000 experiments for each system with dimensions ranging from 1 to 20. 
Figures \ref{fig:Mode_1} and \ref{fig:Mode_2} show how error levels vary across the two system modes as both dimensionality and trial count increase. Overall, it can be observed that within 1000 trials at the same dimension, the algorithm is stable and effective, with comparable error levels. As the system dimension increases, the error tends to rise, which can be attributed to the accumulation of errors from more dimensions. However, the overall error remains within an acceptable range.

\clearpage
\printbibliography

\end{document}